\documentclass[11pt]{amsart}
\usepackage{amsmath,amscd,amssymb}
\usepackage{pstricks,pst-node}
\let\oldsection=\section
\renewcommand{\section}[1]{\vspace{.18in}\par\noindent
\addtocounter{section}
{1}\setcounter{subsection}{0}{\bf\thesection\hspace{9pt}#1}}

\renewcommand{\subsection}{\par\vspace{.18in}\noindent\addtocounter
{subsection}{1}\setcounter{equation}{0}{\bf\thesubsection\hspace{9pt}}}

\theoremstyle{plain}
\newtheorem{thm}{Theorem}
\newtheorem{cor}{Corollary}
\newtheorem{lem}{Lemma}
\newtheorem{prop}{Proposition}

\theoremstyle{definition}

\newtheorem{rem}{Remark}

\newtheorem{obs}{Observation}

\numberwithin{equation}{subsection}

\newcommand{\gfp}{G(\mathbb{F}_p)}
\newcommand{\gfpr}{G(\mathbb{F}_{q})}
\newcommand{\Fr}{{\text{\rm Fr}}}
\newcommand{\Fp}{\mathbb{F}_{p}}

\newcommand{\cgr}{\mathcal{G}_r(k)}
\newcommand{\ul}{\mathfrak{u}}

\newcommand{\Lie}{\operatorname{Lie}}

\newcommand{\ch}{\text{\rm ch}}
\newcommand{\ind}{\operatorname{ind}}
\newcommand{\Ext}{\operatorname{Ext}}
\newcommand{\into}{\hookrightarrow}

\newcommand{\opH}{\operatorname{H}}

\newcommand{\Hom}{\operatorname{Hom}}

\renewcommand{\mod}{\operatorname{mod}}
\renewcommand{\ch}{\operatorname{char}}

\newcommand{\ga}{\gamma}

\newcommand{\la}{\lambda}

\newcommand{\al}{\alpha}
\newcommand{\ta}{\tilde{\alpha}}
\newcommand{\be}{\beta}

\newcommand{\si}{\sigma}

\begin{document}

\title[On the vanishing ranges for the cohomology of finite groups of Lie type]
{\bf On the vanishing ranges for the cohomology of finite groups of Lie type}

\author{\sc Christopher P. Bendel}
\address
{Department of Mathematics, Statistics and Computer Science\\
University of
Wisconsin-Stout \\
Menomonie\\ WI~54751, USA}
\thanks{Research of the first author was supported in part by NSF
grant DMS-0400558}
\email{bendelc@uwstout.edu}

\author{\sc Daniel K. Nakano}
\address
{Department of Mathematics\\ University of Georgia \\
Athens\\ GA~30602, USA}
\thanks{Research of the second author was supported in part by NSF
grant DMS-0654169}
\email{nakano@math.uga.edu}

\author{\sc Cornelius Pillen}
\address{Department of Mathematics and Statistics \\ University of South
Alabama\\
Mobile\\ AL~36688, USA}
\email{pillen@jaguar1.usouthal.edu}

\date{June 2010}

\begin{abstract} Let $G({\mathbb F}_{q})$ be a finite Chevalley group 
defined over the field of $q=p^{r}$ elements, and $k$ be an algebraically 
closed field of characteristic $p>0$. A fundamental open and elusive problem 
has been the computation of the cohomology ring $\opH^{\bullet}(G({\mathbb F}_{q}),k)$. 
In this paper we determine initial vanishing ranges which improves upon known results. 
For root systems of type $A_n$ and $C_n$, the first non-trivial cohomology classes are 
determined when $p$ is larger than the Coxeter number (larger than twice the Coxeter number for type $A_n$ with $n>1$ and $r >1$). In the process we make  
use of techniques involving line bundle cohomology for the flag variety $G/B$ and its relation 
to combinatorial data from Kostant Partition Functions. 
\end{abstract}

\maketitle


\section{Introduction}

\subsection Let $G$ be a simple algebraic group over an algebraically closed 
field $k$ of prime characteristic $p>0$ which is split over the prime field $\Fp$. Let  
$\Fr : G \to G$ denote the Frobenius map and set $q=p^{r}$. The fixed points of the $r$th iterate of 
the Frobenius map, denoted $\gfpr$, is a finite Chevalley group. An outstanding 
open problem of major interest for algebraists and topologists has been 
to determine the cohomology ring $\opH^{\bullet}(\gfpr,k)$\footnote{In the cross-characteristic 
situation (i.e., if $\text{char }k=l$ and $\gcd(l,p)=1$) much more is known about the cohomology of 
$\opH^{\bullet}(\gfpr,k)$ (cf. \cite[Chapter VII]{AM}). In fact in many cases in the non-describing 
characteristic the cohomology (including the ring structure) is completely determined.}. In 2005, during 
a talk at an Oberwolfach conference, Friedlander mentioned that so little is known about this computation that 
it is not even known in which degree the first non-trivial cohomology class occurs.

Our paper aims to address this fundamental question by investigating two problems: 
\vskip .15cm 
\noindent 
(1.1.1) Determining Vanishing Ranges: Finding $D>0$ such that 
the cohomology groups $\opH^i(\gfpr,k)=0$ for $0< i < D$.
\vskip .15cm 
\noindent 
(1.1.2) Locating the First Non-Trivial Cohomology Class: In many instances in conjunction 
with the aforementioned problem, we will find a $D$ such that $\opH^i(\gfpr,k)=0$ for $0<i < D$ and 
$\opH^{D}(\gfpr,k)\neq 0$. A $D$ satisfying this property will be called a {\em sharp bound}. 
\vskip .15cm 
There have been earlier results in the 1970s and 80s addressing (1.1.1). Quillen \cite{Q} showed 
that $\opH^i(GL_n(\mathbb{F}_{q}),k) = 0$ for all $0 < i < r(p - 1)$ and all $n$. 
In that work, he noted that the arguments showed for any $G$ as above, there exists a 
constant $C$ depending on the root system such that $\opH^i(\gfpr,k) = 0$ for $0 < i < r\cdot C$. 
However, no explicit value of $C$ is given except for $G = SL_2$ (and $p$ odd) in which case 
one can take $C = (p-1)/2$.  Furthermore, it was not determined whether these vanishing ranges were 
sharp.  Indeed, in the case of $SL_2$, one can see from work of Carlson \cite{C} that these 
bounds are not sharp in general.  Quillen's original work arose in the context of certain 
$K$-theory computations. Friedlander \cite{F} later used $K$-theoretic techniques to find vanishing 
ranges for more general reductive groups.  Later work of Hiller \cite{H} 
extended Friedlander's result and found vanishing ranges for groups of all types.

Friedlander and Parshall \cite[(A.1) Lemma]{FP} found a sharp bound for the Borel subgroup $B(\mathbb{F}_{q})$
of $GL_n (\mathbb{F}_{q})$. Independent of this work, Barbu \cite{B} constructed a non-zero cohomology class 
in $\opH^{2p-2}(GL_n({\mathbb F}_{p}),k)$ for $p\geq n$. In this paper he conjectured that the sharp 
bound is $D=2p-3$ for $GL_n({\mathbb F}_{p})$ when $n\geq 2$ and $p\geq 3$ 
(cf. \cite[Section 1, Conjecture 4.11]{B}). Since that time, few if any results have been 
obtained in this direction.


\subsection{} The strategy in addressing (1.1.1) and (1.1.2) will entail using 
new and powerful techniques developed by the authors which relate $\opH^i(\gfpr,k)$ to 
extensions over $G$ via a truncated version of the induction functor 
(cf. \cite{BNP1, BNP2, BNP3, BNP5, BNP6}). An outline of the overall strategy is 
presented in the diagram below. For the purposes in this paper 
we will use a non-truncated induction functor ${\mathcal G}_{r}(-)$. We demonstrate that 
when applied to the trivial module $k$, ${\mathcal G}_{r}(k)$ has a filtration 
with factors of the form $H^0(\la) \otimes H^{0}(\la^*)^{(r)}$ (cf. Proposition 2.4). 
The $G$-cohomology of these factors can be analyzed by using the Lyndon-Hochschild-Serre 
(LHS) spectral sequence involving the first Frobenius kernel $G_{r}$ (cf. Section 3.1). 
In particular for $r=1$, we can apply the results of Kumar-Lauritzen-Thomsen \cite{KLT} 
to bound the dimension of the cohomology group 
$\opH^{\bullet}(G({\mathbb F}_{p}),k)$ from above
(cf. Theorem 3.3). The upper bound on the dimension  involves the combinatorics of the well-studied 
Kostant Partition Function. This reduces the question of the vanishing of the finite group 
cohomology to a question involving the combinatorics of the underlying root system $\Phi$.

\begin{picture}(470,110)(0,0)
\put(0,50){$\text{H}^i(G({\mathbb F}_{q}),k)$}
\put(70,50){$ \Longrightarrow$}
\put(55,82){$ \text{ Induction}$}
\put(59,70){$ \text{ Functor}$}
\put(100,50){$\opH^i(G,\cgr)$}
\put(60,10){$\text{H}^i(G,H^0(\la) \otimes H(\la^*)^{(r)})$}
\put(195,10){$ \Longrightarrow$}
\put(130,30){$ \Downarrow$}
\put(165,-10){$ \text{ LHS Spectral}$}
\put(150,30){$ \text{ Filtrations}$}
\put(172,-22){$ \text{ Sequences}$}
\put(220,10){$\text{H}^i(G_1,H^0(\la))$}
\put(300,10){$ \Longrightarrow$}
\put(268,-10){$ \text{Kostant Partition}$}
\put(280,-22){$ \text{ Functions}$}
\put(328,10){Root Combinatorics.}
\end{picture}

\vskip 2cm 

More specifically, for a group $G$ of classical type and $\opH^i(\gfpr,k)$, under the assumption 
that $p > h$ (the Coxeter number), in Theorem 4.4, we identify a vanishing range which improves 
upon (in almost all cases) the ranges of \cite{H}.  Furthermore, in type $C_n$ and $A_n$, we identify 
a {\it sharp vanishing bound} which addresses (1.1.2) (cf. Theorems 5.4, 6.13, 6.14).  These  
bounds are established for primes larger than the Coxeter number, except for  
type $A_n$ with $r > 1$ where sharp vanishing bounds are found for $p$ greater 
than twice the Coxeter number. Finally, as a demonstration of the effectiveness of 
our methods we verify Barbu's Conjecture for $G=GL_n({\mathbb F}_{q})$ when $n\geq 2$ and $p\geq n+2$ 
(cf. Theorem 6.15). 

Our results provide a conceptual description of how the geometry of the nilpotent cone 
plays a role in the description of the cohomology  $\opH^i(G({\mathbb F}_{p}),k)$. In particular 
we prove for $p>h$ (cf. Theorem 3.3): 

$$\dim \text{H}^i(G({\mathbb F}_{p}),k) \leq \sum_{\{w \in W | \ell(w) \equiv i \mod 2\}} \sum_{\mu \in X(T)_+} \sum_{u \in W} 
(-1)^{\ell(u)} P_{\frac{i-\ell(w)}{2}}( u\cdot(p\mu+w\cdot 0) - \mu).
$$ 
Here, $P_{n}(-)$ is the Kostant Partition Function. The root combinatorics involving the Kostant 
Partition Function naturally arises in the context of composition factor multiplicities in 
the ring of regular functions on the nilpotent cone ${\mathcal N}$ of ${\mathfrak g}=\text{Lie }G$ 
(cf. \cite{Jan5} \cite{Br}). This result reinforces  
work of Carlson, Lin and Nakano \cite{CLN} where they prove that the spectrum of this cohomology 
ring is given by the coordinate algebra on ${\mathcal N}^{{\mathbb F}_{p}}/G({\mathbb F}_{p})$ 
where ${\mathcal N}^{{\mathbb F}_{p}}$ is the variety inside of ${\mathcal N}$ consisting of 
${\mathbb F}_{p}$-expressible elements.

The sections of the paper are outlined as follows. In Section 2, we review our previous work and develop the 
necessary cohomological tools related to induction functors which will be used to determine vanishing ranges.  
In Section 3, we present some further cohomological properties relating extensions over 
$G$ with those over the Frobenius kernel $G_r$.  In Section 4, our general vanishing 
bounds are presented.  Finally, Sections 5 and 6 deal with the special cases of root systems 
of types $C_n$ and $A_n$ respectively. The reader might be surprised to see type $C$ treated prior to type $A$. 
It turns out that the root systems of type $C$ are by far the easiest to be dealt with. The 
authors in future work plan to address (1.1.1) and (1.1.2) in the case of the remaining classical groups and the exceptional groups.


\section{Relating $\gfpr$ and $G$}


\subsection{\bf Notation.} Throughout this paper, we will
follow the basic conventions provided in \cite{Jan}.
Let $G$ be a simple simply connected algebraic group scheme which is defined
and split over the finite field ${\mathbb F}_p$ with $p$ elements,
and let $k$ be a field of characteristic $p$.
For $r\geq 1$, let $G_r:=\text{ker }F^{r}$ be the $r$th Frobenius kernel of $G$
and $\gfpr$ be the associated finite Chevalley group.
Let $T$ be a maximal split torus and $\Phi$ be the root system
associated to $(G,T)$. The positive (resp. negative)
roots are $\Phi^{+}$ (resp. $\Phi^{-}$), and $\Delta$ is
the set of simple roots. Let $B$ be a Borel subgroup
containing $T$ corresponding to the negative roots and $U$ be the
unipotent radical of $B$. For a given root system of rank $n$,
the simple roots will be denoted by $\al_1, \al_2, \dots, \al_n$.
We will adhere to the Bourbaki ordering of simple roots. In particular, for type $B_n$,
$\al_n$ denotes the unique short simple root and for type $C_n$,
$\al_n$ denotes the unique long simple root. 
The longest (positive) root will be denoted $\ta$, and for root systems
with multiple root lengths, the longest short root will be denoted $\al_0$.
Let $W$ denote the Weyl group associated to $\Phi$, and, for $w \in W$, let $\ell(w)$
denote the length of the word.

Let ${\mathbb E}$ be the Euclidean space associated with $\Phi$, and
the inner product on ${\mathbb E} $ will be denoted by $\langle\ , \
\rangle$. Let $\alpha^{\vee}=2\alpha/\langle\alpha,\alpha\rangle$ 
be the coroot corresponding to $\alpha\in \Phi$.
In this case, the fundamental weights (basis dual to
$\al_1^{\vee}, \al_2^{\vee}, \dots, \al_n^{\vee}$)
will be denoted by $\omega_1$, $\omega_2$, \dots,
$\omega_n$. Let $X(T)$ be the integral weight lattice spanned by
these fundamental weights. The set of dominant integral weights is
denoted by $X(T)_{+}$. For a weight $\la \in X(T)$, set 
$\la^* : = -w_0\la$ where $w_0$ is the longest word in the 
Weyl group $W$.   By $w \cdot \la := w(\la + \rho) -\rho$  we denote the ``dot" action  of $W$ on $X(T)$, with $\rho$ being the half-sum of the positive roots. For $\al \in \Delta$, $s_{\al} \in W$ denotes the 
reflection in the hyperplane determined by $\al$.

For a $G$-module $M$, let $M^{(r)}$ be the module obtained by composing 
the underlying representation for $M$ with $F^{r}$. Moreover, let 
$M^*$ denote the dual module. For $\la \in X(T)_+$, let $H^0(\la) := \ind_B^G\la$ 
be the induced module and $V(\la) := H^0(\la^*)^*$ be the Weyl module of highest 
weight $\lambda$.


\subsection{} We record two observations on 
roots that will be used at various points in the paper:

\begin{obs}[A]\label{long}
If $\be \in \Phi^{+}$ with $\be \neq \ta$, then 
$\langle\be,\ta^{\vee}\rangle \in \{0,1\}$.
\end{obs}

\begin{obs}[B]\label{unique}
If $w \in W$ admits a reduced expression
$w = s_{\be_1}s_{\be_2}\dots s_{\be_m}$ with $\be_i \in \Delta$ and 
$m = \ell(w)$, then 
$$
-w\cdot 0 = \be_1 + s_{\be_1}(\be_2) + 
s_{\be_1}s_{\be_2}(\be_3) + \cdots + s_{\be_1}s_{\be_2}\dots s_{\be_{m-1}}(\be_m).
$$
Moreover, this is the unique way in which $-w\cdot 0$ can be expressed as a
sum of distinct positive roots.
\end{obs}


\subsection{\bf The Induction Functor.} 
Set $\cgr := \ind_{\gfpr}^G(k)$.  While this $G$-module is infinite dimensional,
it provides a potential way to relate extensions over $\gfpr$ with extensions
over $G$.  Indeed, by Generalized Frobenius Reciprocity and the fact that 
$G/G({\mathbb F}_q)$ is affine, we have the following key observation.

\begin{prop} Let $M, N$ be rational $G$-modules. Then, for all $i \geq 0$,
$$
\Ext^i_{\gfpr}(M,N) \cong \Ext_G^i(M,N\otimes\cgr).
$$

\end{prop}


\subsection{\bf Good Filtrations.} To make the desired computations of 
cohomology groups, we will make use of Proposition 2.3 (with $M = k = N$).  
In addition, we will use a special filtration on $\cgr$.  Recall that
a $G$-module $M$ has a {\em good filtration} if it admits a filtration with 
successive quotients of the form $H^0(\la)$, $\la \in X(T)_+$ \cite[II 4.16]{Jan}.

One may consider $k[G]$ as a $G \times G$-module via the left and right regular actions, 
respectively. A result due to Donkin and Koppinen \cite[II 4.20]{Jan} now says that $k[G]$ as a 
$G\times G$-module admits a good filtration with factors of the form $H^0(\la)\otimes H^0(\la^*)$, 
each $\la \in X(T)_+$ appearing exactly once. Here $(g_1, g_2) \in G \times G$ acts via $g_1 \otimes g_2$ 
on each factor. If one takes the diagonal embedding of $G$ into $G\times G$ one can use this fact to 
show that $k[G]$ admits a good filtration under the adjoint action of $G$.

For our purposes, we modify this slightly by using a partial
Frobenius twist.  Consider the composite
$$
\phi: G \overset{\text{diag}}{\to} G\times G \overset{1\times F^r}{\to} G\times G.
$$
That is, when we take the diagonal embedding, we apply the Frobenius morphism
$r$-times to the second factor giving $\phi(g) = (g,F^r(g))$.  Let $G \times G$
act on $k[G]$ via the left and right regular representations as above, and
then restrict this to a module over $G$ via $\phi$.  Denote the resulting
$G$-module by $k[G]^{\vee}$. The next proposition investigates filtrations 
on $k[G]^{\vee}$.

\begin{prop} As $G$-modules $\cgr \cong k[G]^{\vee}$.  Moreover, $\cgr$
has a filtration with factors of the form $H^0(\la)\otimes H^0(\la^*)^{(r)}$ 
with multiplicity one for each $\la \in X(T)_+$.
\end{prop}

\begin{proof} Let $f \in k[G]$ and $g, x \in G$. If we denote the action of $G$ 
on $k[G]^{\vee}$ by $\star$ then 
\begin{equation*}
(g\star f)(x) = f( g^{-1} x F^r(g)).
\end{equation*}
Next we define  a Lang map $L: G \to G$ via $L(g) = g F^r(g^{-1})$.  By setting  $L^{*}(f) = f\circ L$ one obtains 
a bijection (\cite[1.4]{Hum})
$$L^*: k[G] \to  \cgr = \ind_{\gfpr}^G(k) = \{ f \in k[G] \;| \;f(gh) = f(g) \text{ for all } g\in G , h \in \gfpr \}.$$
Observe that
$$L^{*}(g \star f)(x) = (g \star f) (L(x))= f(g^{-1} L(x) F^r(g))$$
while
\begin{eqnarray*}
(g (L^*(f)))(x)&=& L^*(f)(g^{-1}x) = f(L(g^{-1}x))= f(g^{-1}x F^r(x^{-1}g))\\
& =&f(g^{-1} x F^r(x^{-1})Fr^r(g)) = f(g^{-1}L(x)F^r(g)).
\end{eqnarray*}
Hence, $L^*$ is a $G$-equivariant bijective map from $k[G]^{\vee}$ onto  $\cgr$.

Since $k[G]$, viewed as a $G \times G$-module, has a good filtration with factors $H^0(\la)\otimes H^0(\la^*)$ 
and $\phi(g)$ acts on each factor via $g \otimes F^r(g)$, it follows that $k[G]^{\vee}$ has a filtration with 
factors of the form  $H^0(\la)\otimes H^0(\la^*)^{(r)}$ and multiplicity one for each $\la \in X(T)_+$, as claimed.
\end{proof}


\subsection{} Given weights $\la, \mu \in X(T)$, recall that we say $\mu < \la$ if and only if
$\la - \mu = \sum_{\al \in \Delta} c_{\al}\al$ for integers $c_{\al} \geq 0$. That is,
$\la - \mu$ must lie in the positive root lattice. In addition we say two weights $\la $ and $\mu$ are 
linked if there exists an element $w$ of the affine Weyl group such that $\mu = w \cdot \la$. Note that two 
weights $(\la_1, \la_2), (\mu_1, \mu_2) \in X(T\times T)$ are linked for the algebraic group $G \times G$ 
if and only if the components are linked for $G$.

For each dominant weight $\la$ we define two finite saturated sets, namely
$$\pi_{<\la}= \{ \mu \in X(T)_+ | \;\mu < \la\} \text{  and  }
\pi_{ \leq \la}= \{ \mu \in X(T)_+ |\; \mu \leq \la\}.$$ 
According to \cite[II A.15]{Jan} the $G\times G$-module $k[G]$ has two submodules $M_{< \la}$ and $M_{\leq \la}$, 
both admitting good filtrations with factors $H^0(\nu)\otimes H^0(\nu^*)$ where each $\nu \in \pi_{<\la}$ ($\pi_{ \leq \la}$, 
respectively) appears exactly once. By $S_{< \la}$ and $S_{\leq \la}$ we denote the $G \times G$-summands of  
$M_{< \la}$ and $M_{\leq \la}$, respectively, whose $G \times G$-composition factors have highest weights  contained 
in the same $G \times G$-linkage class as $(\la , \la^*)$.  Similarly we define the quotients $Q_{\nless\la} =k[G]/S_{<\la}$ and
$Q_{\nleq\la} = k[G]/S_{\leq\la}$.

The group $G$ acts on these modules via the embedding $\phi$. From Proposition 2.4 
and \cite[II 4.17]{Jan} one obtains the following result.

\begin{thm} For each $\la \in X(T)_{+}$, there exist short exact sequences of $G$-modules
$$0 \to S_{< \la} \to \cgr \to Q_{\nless \la} \to 0$$  
and 
$$
0 \to S_{\leq \la} \to \cgr \to Q_{\nleq \la} \to 0
$$
with the following properties:
\begin{itemize}
\item[(a)] $S_{< \la}$  ($S_{\leq \la}$) has a 
filtration with factors of the form $H^0(\nu)\otimes H^0(\nu^*)^{(r)}$ where $\nu < \la$ ($\nu \leq \la$) and $\nu $ is linked to $\la$.
\item[(b)] $Q_{\nless \la}$  ($Q_{\nleq \la}$) has a 
filtration with factors of the form $H^0(\nu)\otimes H^0(\nu^*)^{(r)}$ where $\nu \nless \la$ ($\nu \nleq \la$) or $\nu $ is not linked to $\la$.
\item[(c)] The multiplicity in all cases is one.
\end{itemize}
\end{thm}


\subsection{\bf Upper Bounds for $\text{Ext}^i_{\gfpr}$ and  Vanishing Criteria.} The next theorem and its corollaries illustrate the 
usefulness of the existence of the filtrations in Proposition 2.4. 

\begin{thm} Let $M, N$ be rational $G$-modules and $i\geq 0$.  
Then
$$\dim \Ext_{\gfpr}^i (M,N)\leq \sum_{\la \in X(T)_+} \dim\Ext_{G}^i(M,N\otimes H^0(\la)\otimes H^0(\la^*)^{(r)}).$$
\end{thm}

\begin{proof} 
The result  follows immediately from Proposition 2.3, Proposition 2.4 and the long exact sequence in cohomology 
associated to a short exact sequence.
\end{proof}

One obtains the following vanishing criterion. 
\begin{cor}[A] Let $M, N$ be rational $G$-modules and $i\geq 0$.  If 
$\Ext_{G}^i(M,N\otimes H^0(\la)\otimes H^0(\la^*)^{(r)}) = 0$ for all $\la \in X(T)_{+}$,
then $\Ext_{\gfpr}^i(M,N) = 0$.
\end{cor}

In particular, in the special case of $M = k = N,$ we get the following 
criterion for a vanishing range in cohomology. We will see later how to identify such 
an $m$ for a given $G$.

\begin{cor}[B] Let $m$ be the least positive integer such that there exists
$\la \in X(T)_{+}$ with $\opH^m(G,H^0(\la)\otimes H^0(\la^*)^{(r)}) \neq 0$.
Then $\opH^i(\gfpr,k) \cong \opH^i(G,\cgr) = 0$ for $0 < i < m$.  
\end{cor}


\subsection{\bf Non-vanishing.} While the identification of an $m$
satisfying Corollary 2.6(B) gives a vanishing range, it does not a priori
follow that $\opH^m(\gfpr,k) \neq 0$. In this section, we develop 
some conditions under which this conclusion could be made, as well as conditions 
under which one might be able to precisely identify the cohomology group.
As in Corollary 2.6(A), we have the following.

\begin{prop}[A] Let $\la \in X(T)_{+}$ and $M \in \{S_{<\la}, S_{\leq\la},Q_{\nless \la},Q_{\nleq \la}$\}.
Suppose that 
$$\opH^i(G,H^0(\nu)\otimes H^0(\nu^*)^{(r)}) = 0$$ for all $\nu \in X(T)_{+}$
which appear in the filtration for $M$.  Then $\opH^i(G,M) = 0$.
\end{prop}

The next proposition reduces the problem of showing the non-vanishing of 
$\opH^m(G,\cgr)$ (and hence of $\opH^m(\gfpr,k)$) to showing non-vanishing for a 
submodule of $\cgr$.

\begin{prop}[B] Let $m$ be a positive integer. For any $\la \in X(T)_{+}$,
\begin{itemize}
\item[(i)] if $\opH^{m+1}(G,S_{<\la}) = 0$, then the map $\opH^m(G,\cgr) \to \opH^m(G,Q_{\nless \la})$
is surjective;
\item[(ii)] if $\opH^m(G,S_{<\la}) = 0$, then the map $\opH^m(G,\cgr) \to \opH^m(G,Q_{\nless \la})$ 
is injective;
\item[(iii)] if the conditions in (i) and (ii) hold, then $\opH^m(G,\cgr) \cong \opH^m(G,Q_{\nless \la})$.
\end{itemize}
\end{prop}

\begin{proof} Consider the short exact sequence
$$
0 \to S_{<\la} \to \cgr \to Q_{\nless \la} \to 0
$$
and the associated long exact sequence in cohomology
$$
\cdots \to \opH^m(G,S_{<\la}) \to \opH^m(G,\cgr)
        \to \opH^m(G,Q_{\nless \la}) \to \opH^{m+1}(G,S_{<\la}) \to \cdots.
$$
The claims in (i) and (ii) follow immediately from this exact sequence. 
\end{proof}

\begin{rem} In order to show the non-vanishing of $\opH^m(\gfpr,k)$, 
condition (i) is the crucial condition.
Whereas condition (ii) potentially allows us to identify $\opH^m(\gfpr,k)$ precisely.
Note that condition (ii) is immediately satisfied by any weight $\la$ which is 
minimal with respect to the above standard ordering such that
$\opH^m(G,H^0(\la)\otimes H^0(\la^*)^{(r)}) \neq 0$.
\end{rem}

Having reduced the problem to $\opH^i(G,Q_{\nless \la})$, we make some 
similar homological observations about this group.

\begin{prop}[C] Let $m$ be the least positive integer such that there exists
$\nu \in X(T)_{+}$ with $\opH^m(G,H^0(\nu)\otimes H^0(\nu^*)^{(r)}) \neq 0$.  
For any $\la \in X(T)_+$,
\begin{itemize}
\item[(i)] the map $\opH^m(G,H^0(\la)\otimes H^0(\la^*)^{(r)}) \to \opH^m(G,Q_{\nless \la})$ is injective;
\item[(ii)] if $\opH^m(G,Q_{\nleq \la}) = 0$, then $\opH^m(G,H^0(\la)\otimes H^0(\la^*)^{(r)}) \cong 
\opH^m(G,Q_{\nless \la})$.
\end{itemize}
\end{prop}

\begin{proof} Consider the short exact sequence 
$$
0 \to H^0(\la)\otimes H^0(\la^*)^{(r)} \to Q_{\nless \la} \to Q_{\nleq \la} \to 0
$$
and the associated LES
$$
\cdots \rightarrow \opH^{m-1}(G,Q_{\nleq \la}) \to \opH^m(G,H^0(\la)\otimes H^0(\la^*)^{(r)}) \to
        \opH^m(G,Q_{\nless \la}) \to \opH^m(G,Q_{\nleq \la}) \to \cdots.
$$
If $m > 1$, the claims follow immediately since the first term is zero by minimality of $m$.

Suppose that $m = 1$.  Suppose first that $\la$ is linked to the zero weight.  
Since we certainly cannot have $0 > \lambda$, by Theorem 2.5(b),
$\opH^0(G,Q_{\nleq \la}) = \Hom_G(k,Q_{\nleq \la}) = 0$,
and the argument follows as above.

Suppose now that $\la$ is not linked to the zero weight.
The module $Q_{\nless \la}$ may be decomposed as a direct sum  $M_1\oplus M_2$ where 
$M_1$ has a filtration with factors of the form $H^0(\nu)\otimes H^0(\nu^*)^{(1)}$ with
$\nu$ linked to $\la$ and $M_2$ has such a filtration with $\nu$ not linked to $\la$.
Then $\opH^1(G,Q_{\nless \la}) = \opH^1(G,M_1)\oplus \opH^1(G,M_2)$.  Consider the 
short exact sequence
\begin{equation}\label{ses}
0 \to H^0(\la)\otimes H^0(\la^*)^{(1)} \to M_1 \to M_1' \to 0,
\end{equation} 
which defines a module $M_1'$ with a corresponding filtration.
Observe that $Q_{\nleq \la} \cong M_1'\oplus M_2$, and hence
$\opH^1(G,Q_{\nleq\la}) \cong \opH^1(G,M_1')\oplus \opH^1(G,M_2)$.
Since $\la$ is not linked to zero, $\opH^0(G,M_1') = \Hom_G(k,M_1') = 0$.  
Arguing as above with the LES
associated to (\ref{ses}) gives
$$
\opH^1(G,H^0(\la)\otimes H^0(\la^*)^{(1)}) \into \opH^1(G,M_1) \into \opH^1(G,Q_{\nless\la}),
$$
and so part (i) holds.  For part (ii), it follows from the additional assumption that
$\opH^1(G,M_1') = 0$ and $\opH^1(G,M_2) = 0$.  Again using the LES associated to (\ref{ses}),
it follows that
$$
\opH^1(G,H^0(\la)\otimes H^0(\la^*)^{(1)}) \cong \opH^1(G,M_1) = 
\opH^1(G,M_1)\oplus\opH^1(G,M_2) \cong \opH^1(G,Q_{\nless\la})
$$
as claimed.

\end{proof}


\subsection{} Combining the propositions in the preceding section, we can obtain a condition under which 
we have sharp vanishing bounds, and an explicit identification
of $\opH^m(\gfpr,k)$ with a single $G$-cohomology group.

\begin{thm}[A] Let $m$ be the least positive integer such that there exists
$\nu \in X(T)_{+}$ with $\opH^m(G,H^0(\nu)\otimes H^0(\nu^*)^{(r)}) \neq 0$. 
Let $\la \in X(T)_+$ be such that 
$\opH^m(G,H^0(\la)\otimes H^0(\la^*)^{(r)}) \neq 0$.  
Suppose $\opH^{m+1}(G,H^0(\nu)\otimes H^0(\nu^*)^{(r)}) = 0$ for all $\nu < \la$ that are linked to $\la$.
Then
\begin{itemize}
\item[(i)] $\opH^i(\gfpr,k) = 0$ for $0 < i < m$;
\item[(ii)] $\opH^m(\gfpr,k) \neq 0$;
\item[(iii)] if, in addition, $\opH^{m}(G,H^0(\nu)\otimes H^0(\nu^*)^{(r)}) = 0$
for all $\nu \in X(T)_+$ with $\nu \neq \la$, then 
$\opH^m(\gfpr,k) \cong \opH^m(G,H^0(\la)\otimes H^0(\la^*)^{(r)})$.
\end{itemize}
\end{thm}

\begin{proof}  Part (i) follows from Corollary 2.6(B).   For part (ii), note first that, 
by the hypothesis on $\la$ and Proposition 2.7(C)(i),
$\opH^m(G,Q_{\nless \la}) \neq 0$.  On the other hand, by the hypothesis on 
weights less than $\la$ and Proposition 2.7(A), $\opH^{m+1}(G,S_{<\la}) = 0$.  
Hence, by Proposition 2.7(B)(i), the map 
$\opH^m(G,\cgr) \to \opH^m(G,Q_{\nless \la})$ is surjective.  Therefore, by Proposition 2.3,
$\opH^m(\gfpr,k) \cong \opH^m(G,\cgr) \neq 0$ as claimed.

For part (iii), by the added assumption and Proposition 2.7(A), we have
$\opH^m(G,S_{<\la}) = 0$ and $\opH^m(G,Q_{\nleq \la}) = 0$.  By Proposition 2.7(B)(iii)
and Proposition 2.7(C)(ii), we have $\opH^m(G,\cgr) \cong \opH^m(G,Q_{\nless \la}) \cong 
\opH^m(G,H^0(\la)\otimes H^0(\la^*)^{(r)})$ as claimed.
\end{proof}

From the filtration on $\cgr$ in Proposition 2.4, $\opH^i(\gfpr,k) \cong \opH^i(G,\cgr)$
can be decomposed as a direct sum over linkage classes of dominant weights. As 
such, using an analogous argument, a slightly weaker condition for non-vanishing 
can be obtained.

\begin{thm} [B] For a fixed linkage class $\mathcal{L}$, let $m$ be the least 
positive integer such that there exists
$\nu \in \mathcal{L}$ with $\opH^m(G,H^0(\nu)\otimes H^0(\nu^*)^{(r)}) \neq 0$.  Let $\la \in \mathcal{L}$ be such that
$\opH^m(G,H^0(\la)\otimes H^0(\la^*)^{(r)}) \neq 0$.  
Suppose $\opH^{m+1}(G,H^0(\nu)\otimes H^0(\nu^*)^{(r)}) = 0$ for all $\nu < \la$
in $\mathcal{L}$. Then $\opH^m(\gfpr,k) \neq 0$.
\end{thm}


\section{Properties of the Cohomology Groups}

\vskip .25cm 
In Section 2, it was shown that knowledge of cohomology groups of the form
$\opH^i(G,H^0(\nu)\otimes H^0(\nu^*)^{(r)})$ for $\nu \in X(T)_{+}$ 
provides information on $\opH^i(\gfpr,k)$.  In this section, we study these
$G$-cohomology groups and collect a number of useful properties that will be
used throughout the remainder of the paper.

\subsection{\bf Reducing to $G_r$-cohomology.} We will make frequent
use of the following identification of $G$-extensions with $G_r$-cohomology.

\begin{lem} Let $\nu_1, \nu_2 \in X(T)_+$.  
Assume that $\opH^j(G_r,H^0(\nu_1))^{(-r)}$ admits a good filtration for
all $j > 0$. Then for all $j$ 
$$
\opH^j(G,H^0(\nu_1)\otimes H^0(\nu_2^*)^{(r)}) \cong
\Ext_{G}^j(V(\nu_2)^{(r)},H^0(\nu_1)) \cong
\Hom_{G}(V(\nu_2),\opH^j(G_r,H^0(\nu_1))^{(-r)}).
$$
\end{lem}

\begin{proof}  The first isomorphism is immediate.  For the crucial
second isomorphism, consider the Lyndon-Hochschild-Serre spectral sequence
\begin{equation*}\label{lhs}
E_2^{i,j} = \Ext^i_{G/G_r}(V(\nu_2)^{(r)},\opH^j(G_r,H^0(\nu_1))) 
        \Rightarrow \Ext_{G}^{i + j}(V(\nu_2)^{(r)},H^0(\nu_1)).
\end{equation*}
We have
\begin{align*}
E_2^{i,j} &= \Ext^i_{G/G_r}(V(\nu_2)^{(r)},\opH^j(G_r,H^0(\nu_1)))\\
                &\cong \Ext^i_G(V(\nu_2),\opH^j(G_r,H^0(\nu_1))^{(-r)}).
\end{align*}
For $\nu \in X(T)_+$, $i > 0$, and $V$ a $G$-module which admits
a good filtration, we have $\Ext^i_{G}(V(\nu),V) = 0$ 
(cf. \cite[Prop. II 4.16]{Jan}).  By the hypothesis, we conclude that
$E_2^{i,j} = 0$ for all $i> 0$ and the spectral sequence collapses
to a single vertical column. This implies that $E_{2}^{0,j}\cong 
\Ext_{G}^j(V(\nu_2)^{(r)},H^0(\nu_1))$ for all $j$. 
\end{proof}

The assumption that $\opH^j(G_r,H^0(\nu))^{(-r)}$ admits a good filtration
is a long-standing conjecture of Donkin.  For $p > h$ (the Coxeter number
of the root system associated to $G$), this is known for $r = 1$ by
results of Andersen-Jantzen \cite{AJ} and Kumar-Lauritzen-Thomsen \cite{KLT}.
For arbitrary $r$, this is known only for all degrees 
in the case $G = SL_2$. When $r$ is arbitrary and $i=1$ ($p$ arbitrary) or $i=2$ ($p\geq 3$), 
Bendel-Nakano-Pillen verified the assumption by direct computation \cite{BNP4, BNP7}. 
Wright \cite{W} has recently verified the $p=2$, $i=2$ case.

We will apply Lemma 3.1 at several points in the $r = 1$ case for direct applications
to $\gfp$ as well as inductively for dealing with $\gfpr$.  As such, we
will generally assume for the remainder of the paper that $p > h$.


\subsection{\bf Dimensions for $r = 1$.} From Lemma 3.1, to obtain information 
about $\opH^i(G,H^0(\nu)\otimes H^0(\nu^*)^{(1)})$ for $\nu \in X(T)_{+}$,
it suffices to consider $\Hom_G(V(\nu),\opH^i(G_1,H^0(\nu)^{(-1)})$.  
It is well-known that, from block considerations, $\opH^i(G_1,H^0(\nu)) = 0$
unless $\nu = w\cdot 0 + p\mu$ for $w \in W$ and $\mu \in X(T)$.
For $p > h$, from \cite{AJ} and \cite{KLT}, we have
\begin{equation}\label{ind}
\opH^i(G_1,H^0(\nu))^{(-1)} = 
\begin{cases}
\ind_B^G(S^{\frac{i - \ell(w)}{2}}(\ul^*)\otimes\mu) &\text{ if } \nu = w\cdot 0 + p\mu\\
0 & \text{ otherwise,}
\end{cases}
\end{equation}
where $\ul = \Lie(U)$.  Note also that, since $p > h$ and $\nu$ is dominant,
$\mu$ must also be dominant.

For a weight $\nu$ and $n \geq 0$, let $P_n(\nu)$ denote the dimension of the 
$\nu$-weight space of $S^n(\ul^*)$.  Equivalently, for $n > 0$, $P_n(\nu)$ denotes the 
number of times that $\nu$ can be expressed as a sum of exactly $n$
positive roots, while $P_0(0) = 1$. The function $P_{n}$ is often referred to as 
{\em Kostant's Partition Function}. By using \cite[3.8]{AJ}, \cite[Thm 2]{KLT}, Lemma 3.1, 
and (\ref{ind}), we can give an explicit formula for the dimension of 
$ \opH^i(G,H^0(\la)\otimes H^0(\la^*)^{(1)})$. Namely,

\begin{prop} Assume $p >h$. Let $\la = p \mu + w\cdot 0 \in X(T)_+$. 
Then 
$$
\dim \opH^i(G,H^0(\la)\otimes H^0(\la^*)^{(1)}) = \sum_{u \in W} (-1)^{\ell(u)} P_{\frac{i-\ell(w)}{2}}( u\cdot\la - \mu).
$$
\end{prop}


\subsection{} From Proposition 2.3 and Proposition 2.4, we can now deduce the following upper bound 
on the dimensions of the cohomology groups $\opH^i(\gfpr,k)$.

\begin{thm} Assume $p >h$. 
$$\dim \opH^i(\gfp,k) \leq \sum_{\{w \in W | \ell(w) \equiv i \mod 2\}} \sum_{\mu \in X(T)_+} \sum_{u \in W} 
(-1)^{\ell(u)} P_{\frac{i-\ell(w)}{2}}( u\cdot(p\mu+w\cdot 0) - \mu).
$$
\end{thm}


\subsection{\bf Degree Bounds.} From our discussion in Section 2, 
to find vanishing ranges for $\opH^{\bullet}(\gfp,k)$ (or 
$\opH^{\bullet}(\gfpr,k)$ more generally), a first step is to 
try to identify the least positive $i$ such that 
$\opH^i(G,H^0(\la)\otimes H^0(\la^*)^{(1)})$ is non-zero.

Assume that $p > h$ and $\la \in X(T)_{+}$ with
$\opH^i(G,H^0(\la)\otimes H^0(\la^*)^{(1)}) \neq 0$ for some $i > 0$. From 
Lemma 3.1 and the discussion in Section 3.2, we know that
$\la = p\mu + w\cdot 0$ for $w \in W$ and $\mu \in X(T)_+$.   
Observe that if $\mu = 0$, in order for $\la$ to be dominant, 
we must have $\la = 0$.  But, $\opH^i(G,k) = 0$ for $i > 0$.
Since we are interested in cohomology in non-zero degrees, we may
safely assume that $\la, \mu \neq 0$.  Corollary 3.5 below gives
a relationship between $i$ and the weight $\la$.  We first derive
a more general relationship that will be useful in inductive arguments.

\begin{prop} Assume that $p > h$.
Let $\ga_1, \ga_2 \in X(T)_+$, both non-zero,  such that 
$\ga_j = p \delta_j + w_j\cdot 0 $  with $\delta_j \in X(T)_+$ and $w_j \in W$ 
for $j=1, 2$. Assume $\Ext^i_G(V(\ga_2)^{(1)}, H^0(\ga_1)) \neq 0$.
\begin{itemize}
\item[(a)] Let $\si \in \Phi^+$.  If $\Phi$ is of type $G_2$, assume further
that $\si$ is a long root.  Then
$p\langle \delta_2, \si^{\vee}\rangle - \langle \delta_1, \si^{\vee}\rangle 
+ \ell(w_1) + \langle w_2\cdot 0,  \si^{\vee}\rangle \leq i.$
\item[(b)] Let $\ta$ denote the longest root in $\Phi^+$.  Then
$p\langle \delta_2, \ta^{\vee}\rangle - 
        \langle \delta_1, \ta^{\vee}\rangle 
+ \ell(w_1) - \ell(w_2) - 1 \leq i.$ Equality requires that $\gamma_2 -\delta_1 =( (i-l(w_1))/2)\ta$ and
$
\langle-w_2\cdot 0,\ta^{\vee}\rangle = 
\ell(w_2) + 1.
$
\end{itemize}
\end{prop}

\begin{proof} By Lemma 3.1, (\ref{ind}), and Frobenius reciprocity, we have
\begin{align*}
\Ext^i_G(V(\ga_2)^{(1)}, H^0(\ga_1)) &\cong
\Hom_{G}(V(\ga_2),\opH^i(G_1,H^0(\ga_1))^{(-1)})\\
 &\cong 
\Hom_G(V(\ga_2),\ind_B^G(S^{\frac{i - \ell(w_1)}{2}}(\ul^*)\otimes \delta_1))\\
&\cong \Hom_B(V(\ga_2),S^{\frac{i - \ell(w_1)}{2}}(\ul^*)\otimes \delta_1).
\end{align*}
Since this is non-zero, 
$\ga_2 =p \delta_2 + w_2 \cdot  0$ must be a weight of 
$S^{\frac{i - \ell(w_1)}{2}}(\ul^*)\otimes\delta_1.$ 
In other words, $\ga_2 - \delta_1 = p\delta_2 - \delta_1 + w_2\cdot 0$
must be a weight of $S^{\frac{i - \ell(w_1)}{2}}(\ul^*)$.

The vector space $\ul^*$ has a basis of root vectors corresponding to positive
roots. So a homogeneous weight of $S^j(\ul^*)$ is a sum of $j$ not necessarily
distinct positive roots.  Therefore, $\ga_2 - \delta_1$ must be expressible
as a sum of $\frac{i - \ell(w_1)}{2}$ positive roots.  For 
any positive roots $\si_1, \si_2$ (with $\si_2$ being long if $\Phi$ is of
type $G_2$), we have 
$\langle\si_1,\si_2^{\vee}\rangle \leq 2$.  Hence, for $\si \in \Phi^+$,
we have
\begin{equation}
\langle \ga_2 - \delta_1,\si^{\vee}\rangle \leq \frac{i - \ell(w_1)}{2}*2 = 
i - \ell(w_1).
\end{equation}
Substituting $\ga_2 - \delta_1 = p\delta_2 - \delta_1 + w_2\cdot 0$ gives
$$
p\langle \delta_2,\si^{\vee}\rangle - \langle\delta_1,\si^{\vee}\rangle
+ \langle w_2\cdot 0,\si^{\vee}\rangle \leq i - \ell(w_1).
$$
Part (a) immediately follows.

For part (b), 
Note that for $\sigma = \ta$ equality in Equation (3.4.1) can only hold if $\gamma_2 -\delta_1 = ((i-l(w_1))/2)\ta$. This follows from Observation 2.2(A). 
In addition, by Observation 2.2(B), 
$-w_2\cdot 0$ can be expressed uniquely as a sum of precisely $\ell(w_2)$
distinct positive roots.  Since at most one of those roots can be $\ta$,
and $\langle\ta,\ta^{\vee}\rangle = 2$, by Observation 2.2(A), we have
$$
\langle-w_2\cdot 0,\ta^{\vee}\rangle \leq (\ell(w_2) - 1)*1 + 2 = 
\ell(w_2) + 1
$$
which gives part (b). 
\end{proof}


\subsection{} As a special case of Proposition 3.4 we have the following result.

\begin{cor} Assume that $p > h$.
Let $\la = p\mu + w\cdot 0$ be a non-zero dominant weight with
$0 \neq \mu \in X(T)_{+}$ and $w \in W$.   
Assume $\opH^i(G,H^0(\la)\otimes H^0(\la^*)^{(1)}) \neq 0$.
\begin{itemize}
\item[(a)] Let $\si \in \Phi^+$.  If $\Phi$ is of type $G_2$, assume further
that $\si$ is a long root. Then
$(p - 1)\langle \mu, \si^{\vee}\rangle + \ell(w) + 
        \langle w\cdot 0,\si^{\vee}\rangle \leq i.$
\item[(b)] Let $\ta$ denote the longest root in $\Phi^+$.  Then
$(p - 1)\langle \mu, \ta^{\vee}\rangle  - 1 \leq i.$ Equality requires that $\la -\mu = ((i-l(w))/2)\ta$ and 
$
\langle-w\cdot 0,\ta^{\vee}\rangle = 
\ell(w) + 1.
$

\end{itemize}
\end{cor}

\begin{proof} Parts (a) and (b) follow immediately from Propositon 3.4 
by taking $\ga_1 = \la = \ga_2$.  
\end{proof}

In the corollary, since $\mu$ is a non-zero dominant
weight, $\langle\mu,\ta^{\vee}\rangle \geq 1$.  Hence,
we immediately have that $i \geq p - 2$.  It follows from Corollary 2.6(B) 
that $\opH^i(\gfp,k) = 0$ for $0 < i < p - 2$.  This will follow
as a special case of a more general result in the next section.


\section{A Minimal Vanishing Range}

\subsection{} In this section, we use the preceding techniques to determine 
a general vanishing range for $\opH^i(\gfpr,k)$ for $p > h$. 
We begin with some further extension properties that will be used
in the proof.

\begin{lem}
Assume that $p > h$ and $r >1$. Let $\la, \mu \in X(T)_+$, both non-zero, 
and $i\geq 0$. If $\Ext^i_G(V(\la)^{(r)}, H^0(\mu)) \neq 0$, then there 
exists a non-zero weight
$\ga \in X(T)_+$ and nonnegative integers $k, l$ such that 
\begin{itemize}
\item[(a)] $i= k+l$,
\item[(b)] $\Ext^k_G(V(\la)^{(r-1)}, H^0(\ga)) \neq 0$,
\item[(c)] $\Ext^l_G(V(\ga)^{(1)}, H^0(\mu)) \neq 0$, and
\item[(d)] $\ga = p \delta + w\cdot 0$ for some $w \in W$ and  
non-zero $\delta \in X(T)_+$.
\end{itemize}
\end{lem}

\begin{proof} Consider the Lyndon-Hochschild-Serre spectral sequence
$$E_{2}^{k,l}=\Ext_{G/G_1}^k(V(\la)^{(r)}, \opH^l(G_1,H^0(\mu))) 
\Rightarrow \Ext^{k+l}_G(V(\la)^{(r)}, H^0(\mu)).$$
The assumptions imply that there exist nonnegative integers $k,l$ with $k+l=i$ and 
 $\Ext_{G/G_1}^k(V(\la)^{(r)}, \opH^l(G_1,H^0(\mu)))\neq 0.$ The $G$-module
$\opH^l(G_1,H^0(\mu))^{(-1)}$ has a good filtration. Therefore, there exists 
a dominant weight $\ga$ with 
 \begin{equation}
 \Hom_{G/G_1}(V(\ga)^{(1)},\opH^l(G_1,H^0(\mu))) \cong
        \Hom_G(V(\ga),\opH^l(G_1,H^0(\mu))^{(-1)})\neq 0
 \end{equation}
 and  
 \begin{equation}
 \Ext_{G}^k(V(\la)^{(r-1)}, H^0(\ga)) \cong
        \Ext_{G/G_1}^k(V(\la)^{(r)}, H^0(\ga)^{(1)}) \neq 0.
\end{equation}
Now  (4.1.2) implies that $\ga$ is of the form $\ga = p \delta + w\cdot 0$ with 
$w \in W$ and  $\delta \in X(T)_+$. By Lemma 3.1 and (4.1.1), 
$\Ext_{G}^l(V(\ga)^{(1)},H^0(\mu)) \cong
\Hom_{G}(V(\ga),\opH^l(G_1,H^0(\mu))^{(-1)}) \neq 0$.  This forces $\delta \neq 0$. 
Note that the assumptions also force $\mu = p \delta' + w'\cdot 0$ for some 
$w' \in W$ and nonzero $\delta' \in X(T)_+$.
\end{proof}


\subsection{} Applying the lemma repeatedly immediately implies the following proposition.

\begin{prop}
Assume that $p > h$. Let $\la, \mu \in X(T)_+$, both non-zero, and $i\geq 0$. If 
$\Ext^i_G(V(\la)^{(r)}, H^0(\mu)) \neq 0$, then there exists a sequence of non-zero weights\\
 $\mu = \ga_0, \ga_1, \dots, \ga_{r-1} , \ga_r= \la \in X(T)_+$ and nonnegative integers $l_1, l_2, \dots, l_r$ such that 
\begin{itemize}
\item[(a)] $i= \sum_{j=1}^rl_j$,
\item[(b)] $\Ext^{l_j}_G(V(\ga_j)^{(1)}, H^0(\ga_{j-1})) \neq 0$, for $1\leq j \leq r$ and 
\item[(c)] $\ga_j = p \delta_j + u_j\cdot 0$ with $u_j \in W$ and  nonzero $\delta_j \in X(T)_+$, for $1\leq j \leq r-1$.
\end{itemize}
\end{prop}


\subsection{} The next step in our analysis is to obtain tighter control over 
a lower bound on $i$ as in Proposition 4.2 in the case that $\la = \mu$.

\begin{prop} Assume that $p > h$.  Let $0 \neq \la \in X(T)_+$ and $i \geq 0$.
If $\opH^i(G,H^0(\la)\otimes H^0(\la^*)^{(r)}) \neq 0$, then there exists a sequence of non-zero weights
 $\la = \ga_0, \ga_1, \dots, \ga_{r-1} , \ga_r= \la \in X(T)_+$  such that 
 $\ga_j = p \delta_j + u_j\cdot 0$ for some $u_j \in W$ and  nonzero $\delta_j \in X(T)_+$.  Furthermore,
\begin{equation}\label{rbound}
i \geq  \left(\sum_{j=1}^r (p-1) \langle \delta_j, 
        \ta^{\vee}\rangle\right) - r.
\end{equation}
Equality requires that $p\delta_j -\delta_{j-1} + u_j\cdot 0= ((l_j-l(u_{j-1}))/2)\ta$ and that
$
\langle-u_j\cdot 0,\ta^{\vee}\rangle = 
\ell(u_j) + 1
$ for all $1\leq j \leq r$, where $l_j$ is as in Proposition 4.2.

\end{prop}

\begin{proof} The first part is simply a partial restatement of Proposition 4.2
with $\la = \mu$. Specifically, there exists
a sequence of non-zero dominant weights 
v$\la = \ga_0, \ga_1, ... , \ga_{r-1}, \ga_r = \la$ with 
$\ga_j = p \delta_j + u_j\cdot 0$ and corresponding nonnegative integers $l_j$ 
with $i= \sum_{j=1}^rl_j$ and 
$\Ext^{l_j}_G(V(\ga_j)^{(1)}, H^0(\ga_{j-1})) \neq 0$.

For (\ref{rbound}), we use Proposition 3.4(b) to obtain the inequalities 
$$p\langle \delta_j, \ta^{\vee}\rangle - \langle \delta_{j-1}, 
\ta^{\vee}\rangle + \ell(u_{j-1}) -\ell(u_j) -1 \leq l_j \mbox{ for } 
1\leq j \leq r,$$
with equality only if $p\delta_j -\delta_{j-1} + u_j\cdot 0= ((l_j-l(u_{j-1}))/2)\ta$ and 
$
\langle-u_j\cdot 0,\ta^{\vee}\rangle = 
\ell(u_j) + 1
$ for all $1\leq j \leq r$.
Note that $\delta_0 = \delta_r$ and $u_0 = u_r$.  Summing over $j$ yields 
\begin{equation*}
 \left(\sum_{j=1}^r (p-1) \langle \delta_j, 
        \ta^{\vee}\rangle\right) -r \leq \sum_{j=1}^r l_j = i.
\end{equation*}
\end{proof}

\subsection{} For $p>h$ we can now present general vanishing ranges which address (1.1.1).

\begin{thm} Assume that  $p > h$. 
Then
\begin{itemize}
\item[(a)] $\opH^i(G,H^0(\la)\otimes H^0(\la^*)^{(r)}) = 0$ for $0 < i < r(p - 2)$ 
        and $\la \in X(T)_+$;
\item[(b)] $\opH^i(\gfpr,k) = 0$ for $0 < i < r(p - 2)$.
\end{itemize}
\end{thm}

\begin{proof}
Part (a) implies part (b) via Corollary 2.6(B). Suppose that 
$$\opH^i(G,H^0(\la)\otimes H^0(\la^*)^{(r)}) \neq 0$$ 
for some $0 < i$ and $\la \in X(T)_+$. Clearly $\la \neq 0$, so we may apply Proposition 4.3.
Since $\ga_j \neq 0$, $1 \leq \langle \delta_j, \ta^{\vee}\rangle.$
Proposition 4.3 then gives $i \geq r(p-1) - r = r(p-2)$ as claimed.
\end{proof}

Observe that this vanishing range is generally larger than the one
obtained in \cite{H}.  Precisely, the ranges obtained in \cite{H} are of the 
form $0 < i < m$ where $m$ depends on the root system.  Except in certain
type $A_n$ cases, $m \leq r(p-1)/2$.

In the remainder of the paper, we further investigate this question 
to determine sharp bounds for root systems of type $C_n$ (for all $r$; see Theorem 5.4)
and $A_n$ (for $r = 1$  and generically for all $r$; see Theorems 6.13, 6.14).  
In type $C_n$, the above bounds are in fact sharp.

\begin{rem} Note that the assumption $\Ext^i_G(V(\la)^{(r)}, H^0(\mu)) \neq 0$ in 
Proposition 4.2 can be replaced by $\Ext_{G/G_1}^k(V(\la)^{(r)}, 
\opH^l(G_1,H^0(\mu)))\neq 0$, where $k+l =i$. In that case one arrives at the same 
conclusions with $l_1 =l$. Now the arguments used to prove Proposition 4.3 and 
Theorem 4.4  can be used to show that 
$\Ext_{G/G_1}^k(V(\la)^{(r)}, \opH^l(G_1,H^0(\la))= 0$ for all $k+l < r(p-2)$.
\end{rem}


\section{Type $C_n$, $n \geq 1$}

\vskip .25cm 
Assume throughout this section that $\Phi$ is of type $C_n$, $n\geq 1$, 
and $p > h = 2n$.


\subsection{\bf Realization for $r=1$.} We determine the least $i > 0$ such 
that $\opH^i(G,H^0(\la)\otimes H^0(\la^*)^{(1)}) \neq 0$. From Theorem 4.4, 
we know that $i\geq p-2$.  Let $\ta = 2\omega_1$ denote the longest 
positive root. 
We next construct a weight $\la$ with 
$\opH^{p-2}(G,H^0(\la)\otimes H^0(\la^*)^{(1)}) \neq 0$.   
Let $w =s_{\ta}= s_1s_2\dots s_{n-1}s_ns_{n-1}\dots s_2s_1 \in W$.  Then
$-w\cdot 0 = n\ta = 2n\omega_1$. Furthermore, when expressed as a sum of distinct
positive roots, $-w\cdot 0$ consists of precisely all positive
roots which contain an $\al_1$.  Set 
$\la = p\omega_1 + w\cdot 0 = p\omega_1 - 2n\omega_1 = 
(p - 2n)\omega_1$.  Then 
$$\la - \omega_1 = (p - 1 - 2n)\omega_1 = 
\left(\frac{p-1}{2} - n\right)\ta$$ 
is a highest weight of $S^j(\ul^*)$ where $j = \frac{p-1}{2} - n$.

We will apply Proposition 3.2 to compute 
$\dim\opH^{p-2}(G,H^0(\la)\otimes H^0(\la^*)^{(1)})$.
Specifically, we will show that for $u \in W$
$$P_{\frac{p-1}{2}-n}( u\cdot\la - \omega_1)=
\begin{cases}1 \mbox{ if } u = 1\\
0 \mbox{ else.}
\end{cases}$$
We will work with the $\epsilon$-basis of $X(T)$. Rewrite  
$u \cdot \la - \omega_1 = ((p-2n)u-1)\epsilon_1 + u \cdot 0$ as 
$\sum_i c_i \epsilon_i$. In order for this expression to be a sum of positive 
roots, the coefficient $c_1$ has to be nonnegative. This forces $u(\epsilon_1) = 
\epsilon_1$. Then $u\cdot 0$ is of the form $-\sum_i d_i \alpha_i$ with $d_1 =0$. 
This implies  that  
$u \cdot \la - \omega_1 = (\frac{p-1}{2}-n){\ta} -\sum_i d_i \alpha_i$. 
Such an expression contains $p-1-2n$ copies of $\alpha_1$. Since $\ta$ 
is the only positive root containing $ 2\alpha_1$, the above expression can  be 
written as a sum of $\frac{p-1}{2}-n$ positive roots  if and only if $u=1$. From 
Proposition 3.2, one concludes that 
$\opH^{p-2}(G,H^0(\la)\otimes H^0(\la^*)^{(1)}) \cong k$.


\subsection{\bf More Vanishing in Degree $p-2$.} To get a precise vanishing range, 
we need to consider whether there are any other weights with non-zero
cohomology in degree $p - 2$. Let $\la = p\mu + w\cdot 0 \in X(T)_{+}$ with 
$\opH^i(G,H^0(\la)\otimes H^0(\la^*)^{(-1)}) \neq 0$
for some $i > 0$.  Consider the maximal short root $\al_0 = \omega_2$. 
By Corollary 3.5(a) with $\si = \al_0$,
\begin{equation}\label{Cone}
i \geq (p - 1)\langle\mu,\al_0^{\vee}\rangle + \ell(w) + 
        \langle w\cdot 0,\al_0^{\vee}\rangle.
\end{equation}
There are three positive roots $\be$ with $\langle\be,\al_0^{\vee}\rangle = 2$
(unless $n = 1$, in which case there is only one and $\al_0 = 2\omega_1$).  
Since $-w\cdot 0$ can be expressed uniquely as a sum of $\ell(w)$ distinct
positive roots, 
we can say $\langle - w \cdot 0,\al_0^{\vee}\rangle \leq \ell(w) + 3$.
Hence, (\ref{Cone}) can be rewritten as 
\begin{equation}\label{Ctwo}
i \geq (p - 1)\langle\mu,\al_0^{\vee}\rangle - 3.
\end{equation}

In Type $C_n$, 
$\langle\omega_1,\al_0^{\vee}\rangle  = 1$. But, for $2 \leq j \leq n$,
$\langle\omega_j,\al_0^{\vee}\rangle = 2$. Hence, for $\mu \in X(T)_{+}$,
if $\mu \neq 0, \omega_1$, then $\langle\mu,\al_0^{\vee}\rangle \geq 2$.
If $\langle\mu,\al_0^{\vee}\rangle \geq 2$, then (\ref{Ctwo}) becomes
$$
i \geq 2(p-1) - 3 = 2p - 5 > p - 2
$$
since $p \geq 5$ ($n \geq 2$).  Therefore, the only candidates for
a non-zero cohomology group in degree $p - 2$  are with 
$\la = p\omega_1 + w\cdot 0$ for some $w \in W$.  This makes sense because 
the weight constructed in Section 5.1 is of this form.


\subsection{\bf A Sharp Bound for $r=1$.} Suppose 
$\la = p\omega_1 + w\cdot 0 \in X(T)_+$ and 
$$\opH^{p-2}(G,H^0(\la)\otimes H^0(\la^*)^{(1)}) \neq 0.$$
Proposition 3.4 part (b) implies that $\la - \omega_1 = (p - 1)\omega_1 + w\cdot 0=((p - 1)/2)\ta + w\cdot 0=((p - 2 - \ell(w))/2)\ta$ and $\langle w\cdot0,\ta^{\vee}\rangle = - (\ell(w) + 1).$ This forces $w\cdot 0 = -((l(w)+1)/2)\ta$. The only possible choices for $w$ satisfying this last equation are $w=0$ and $w=s_{\ta}$. Now $\langle w\cdot0,\ta^{\vee}\rangle = - (\ell(w) + 1)$ forces $w=s_{\ta}$ and $w\cdot0 = -n\ta$.
Hence, $\la = (p-2n)\omega_1,$ the weight given 
in Section 5.1. So the $\la$ exhibited there is the only dominant weight
with $\opH^{p-2}(G,H^0(\la)\otimes H^0(\la^*)^{(1)}) \neq 0$.

Note that $\la = (p-2n)\omega_1$ is contained in the lowest alcove. 
There are no smaller dominant weights that are linked to $\la$ (so the condition 
in Theorem 2.8(A) involving $\opH^{m+1}$ is vacuous). Consequently, 
Theorem 2.8(A) and the above discussion now yields:

\begin{thm} Suppose $\Phi$ is of type $C_n$ with $p > 2n$. Then
\begin{itemize}
\item[(a)] $\opH^i(\gfp,k) = 0$ for $0 < i < p - 2$;
\item[(b)] $\opH^{p-2}(\gfp,k) \cong k$.
\end{itemize}
\end{thm}


\subsection{\bf A Sharp Bound for all $r$.} In this section we will address 
(1.1.1) and (1.1.2) in general for $\opH^{\bullet}(\gfpr,k)$ when $\Phi$ is of 
type $C_{n}$.

\begin{lem}
Assume $p >2n$. Let $\la = (p-2n)\omega_1$. Then 
$\Ext^{r(p-2)}_G(V(\la)^{(r)}, H^0(\la)) \cong k.$
\end{lem}

\begin{proof}
We use induction on $r$. If $r=1$ the assertion follows from Section 5.1. 
Next we make use of the LHS spectral sequence 
$$E_2^{k,l}=\Ext_{G/G_1}^k(V(\la)^{(r)}, \opH^l(G_1,H^0(\la)) \Rightarrow \Ext^{k+l}_G(V(\la)^{(r)}, H^0(\la)).$$ 
From now on we  assume that  $E_2^{k,l}\neq 0$. We apply Lemma 4.1 and Remark 4.4 to conclude that there exists a dominant weight $\gamma=p\delta + u\cdot 0$ with
$\Hom_{G}(V(\ga), \opH^{l}(G_1,H^0(\la))^{(-1)})\neq 0$ and $\Ext^{k}_G(V(\la)^{(r-1)}, H^0(\ga))\neq 0.$
Furthermore, by Propositions 4.2 and  4.3
there exists a sequence of non-zero weights
 $\la = \ga_0, \ga_1, \dots, \ga_{r-1} , \ga_r= \la \in X(T)_+$  with $\ga_1 = \ga$ such that 
 $\ga_j = p \delta_j + u_j\cdot 0$ for some $u_j \in W$ and  nonzero $\delta_j \in X(T)_+$.  In addition,
\begin{equation}
k+l \geq  \left(\sum_{j=1}^r (p-1) \langle \delta_j, 
        \ta^{\vee}\rangle\right) - r.
\end{equation}
Equality requires that $p\delta_j -\delta_{j-1} + u_j\cdot 0= ((l_j-l(u_{j-1}))/2)\ta$ and that
$
\langle-u_j\cdot 0,\ta^{\vee}\rangle = 
\ell(u_j) + 1
$ for all $1\leq j \leq r$, where $l_j$ is as in Proposition 4.2.
It follows immediately that 
$E_2^{k,l} =0$ whenever 
$k+l < r(p-2)$.

Looking  at $\Hom_{G}(V(\ga), \opH^{l}(G_1,H^0(\la))^{(-1)})
\cong \Hom_B(V(\ga),S^{\frac{l-2n+1}{2}}(\ul^*)\otimes\omega_1))$ one concludes that 
for $l\leq p-2$, all weights $\mu$   appearing in $S^{\frac{l-2n+1}{2}}(\ul^*)\otimes\omega_1$  
satisfy $\langle \mu+ \rho, \alpha_0^{\vee} \rangle < p$. Hence, 
$\opH^{l}(G_1,H^0(\la))^{(-1)} $ is completely reducible  for $l \leq p-2$. 
 From Sections 5.1 and 5.2 one concludes that $H^0(\la)$ appears as  a summand if and only if $l= p-2$. Clearly the trivial module does not appear as a summand of $\opH^{l}(G_1,H^0(\la))^{(-1)} $. But these are the only composition factors of 
 $ \opH^{l}(G_1,H^0(\la))^{(-1)}$ that could be linked to
the weight $p^{r-1}\la.$ 
 The linkage principle now forces  $l \geq p-2$. Moreover, if $l=p-2$, the only possible choice for $\gamma$ is that $\gamma = \la$ and hence $E_2^{k,p-2}
 \cong \Ext^{k}_G(V(\la)^{(r-1)}, H^0(\la)).$


If $k+l= r(p-2)$ then (5.4.1) becomes an equality. This forces $\ga _1=\ga=p\delta+u\cdot 0=((l-l(u))/2)\ta+ \omega_1= (l-2n+2)\omega_1$, $\delta_1=\delta = \omega_1$, and $
\langle-u\cdot 0,\ta^{\vee}\rangle = 
\ell(u) + 1
$. 
Using a similar argument to the one in  Section 5.3 one concludes that $\ga = p\omega_1+ s_{\ta}\cdot 0=\la$, which forces $l=p-2$ and $k=(r-1)(p-2)$.

To summarize, we have shown that
$$
E_2^{k,l} \cong
 \begin{cases}
 0 &\text{ if } k+l < r(p-2)\\
 0 &\text{ if } l < p-2\\
 0 &\text{ if } k+l = r(p-2) \text{ and } l\neq p-2\\
\Ext^{k}_G(V(\la)^{(r-1)}, H^0(\la)) &\text{ if } k+l = r(p-2) \text{ and } l=p-2.
\end{cases}
$$ 
 Therefore, the $((r-1)(p-2),p-2)$-term of the $E_2$-page transgresses to the $E_{\infty}$-page and 
produces an isomorphism $\Ext^{r(p-2)}_G(V(\la)^{(r)}, H^0(\la))\cong \Ext^{(r-1)(p-2)}_G(V(\la)^{(r-1)}, H^0(\la))$, and the claim follows by
induction.
\end{proof}

By applying Theorem 4.4, the fact that $\la$ is the smallest weight in its linkage
class, and Theorem 2.8(B) one obtains the following theorem.

\begin{thm} Suppose $\Phi$ is of type $C_n$ with $p > 2n$. Then
\begin{itemize}
\item[(a)] $\opH^i(\gfpr,k) = 0$ for $0 < i < r(p - 2)$;
\item[(b)] $\opH^{r(p-2)}(\gfpr,k) \cong k$.
\end{itemize}
\end{thm}


\section{Type $A_n$, $n \geq 2$}

\vskip .25cm 
Assume throughout this section that $\Phi$ is of type $A_n$, $n \geq 2$, and that
$p > h = n + 1$.  Note that type $A_1$ is equivalent to type $C_1$ which 
was covered in Section 5.


\subsection{\bf An Upper Bound for $r=1$.}  We first construct a weight $\la$ with 
$$\opH^{2p - 3}(G,H^0(\la)\otimes H^0(\la^*)^{(1)}) \neq 0.$$
Set $w :=s_{\ta}= s_1s_2\dots s_{n-1}s_ns_{n-1}\dots s_2s_1 \in W$, where $s_i$ 
is again the reflection corresponding to the $i$th simple root $\al_i$. 
Then $\ell(w) = 2n - 1$ and $-w\cdot 0 = n\ta$.  When decomposed 
uniquely into a sum of distinct positive
roots, $-w\cdot 0$ consists of precisely all positive roots which contain either
an $\al_1$ or an $\al_n$ (or both).
Set $\la := p\ta + w\cdot 0 = (p - n)\ta$ and $\mu := \ta$.
Then $\la - \mu = (p - n - 1)\ta$ is a weight of $S^{p-n-1}(\ul^*)$.
Indeed, it is the highest weight corresponding to taking $(p-n-1)$-copies
of $\phi_{\ta} \in \ul^*$ (the root vector corresponding to $\ta$).  
Similar to the argument in Section 5.1, we will show that 
$$P_{p-n-1}( u\cdot\la - \mu)=
\begin{cases}1 \mbox{ if } u = 1\\
0 \mbox{ else.}
\end{cases}$$
We will work with the $\epsilon$-basis of $X(T)$. Rewrite  
$u \cdot \la - \mu = ((p-n)u-1)(\epsilon_1 - \epsilon_{n+1})) + u \cdot 0$ as 
$\sum_i c_i \epsilon_i$. In order for this expression to be a sum of positive
roots, the coefficient $c_1$ has to be nonnegative and $c_{n+1}$has to be less 
than or equal to zero. This forces $u(\epsilon_1) = \epsilon_1$ and 
$u(\epsilon_{n+1} )= \epsilon_{n+1}$. This forces now $u\cdot 0$ to be of the 
form $-\sum_i d_i \alpha_i$ with $d_1 = d_{n+1}=0$. This implies  that  
$u \cdot \la - \mu = (p-n-1)\ta -\sum_i d_i \alpha_i$ can  be written as 
a sum of $p-n-1$ positive roots  if and only if $u=1$.

\begin{prop} Suppose $\Phi$ is of type $A_n$ with $n \geq 2$ and $p \geq n + 2$.
Let $\la = (p - n)\ta = (p-n)\omega_1 + (p-n)\omega_n$.
Then 
\begin{itemize} 
\item[(a)] $\opH^{2p - 3}(G,H^0(\la)\otimes H^0(\la^*)^{(1)}) \cong k$; 
\item[(b)] $\opH^{2p - 3}(\gfp,k) \neq 0$.
\end{itemize}
\end{prop}

\begin{proof} Part (a) follows from our analysis above and 
Proposition 3.2.

For part (b) suppose that $0 \neq \mu \in X(T)_+$ is linked to $(p - n)\ta$ and
$\opH^i(G,H^0(\mu)\otimes H^0(\mu^*)^{(1)}) \neq 0$ for some $i > 0$.
As noted in Section 3.2, we necessarily have $\mu = p\delta + w\cdot 0$ for some
$0 \neq \delta \in X(T)_+$ and $w \in W$.  Observe that, since $\mu$ lies in the root lattice,
$\delta$ also lies in the root lattice.  Therefore, 
$\langle\delta,\ta^{\vee}\rangle \geq 2$.  From Corollary 3.5(b) we 
get $i \geq 2(p - 1) - 1 = 2p - 3$.

Note also that $(p-n)\ta$ lies in the second fundamental $p$-alcove.  
Indeed, it is the reflection of the zero
weight across the upper wall.  So the only weight $\mu$ of the form 
$p\ta + w\cdot 0$ with $\mu < \la$ would be the zero weight. 
And we know that $\opH^i(G,k) = 0$ for all $i > 0$.  Therefore, we can 
apply Theorem 2.8(B) to deduce the result.

\end{proof}


\subsection{\bf An Upper Bound for $r>1$.} The following result indentifies 
a non-zero cohomology class in degree $r(2p-3)$.

\begin{lem}
Assume $n\geq 2$ and $p >n+2$. Let $\la = (p-n)\ta = (p-n)(\omega_1+ \omega_n)$. Then 
$\Ext^{r(2p-3)}_G(V(\la)^{(r)}, H^0(\la)) \cong k.$
\end{lem}

\begin{proof}
We use induction on $r$. If $r=1$ the assertion follows from Section 6.1. 
The following argument follows closely the argument in Section 5.4. Again we make use of the LHS spectral sequence 
$$E_2^{k,l}=\Ext_{G/G_1}^k(V(\la)^{(r)}, \opH^l(G_1,H^0(\la))) \Rightarrow \Ext^{k+l}_G(V(\la)^{(r)}, H^0(\la)),$$ 
and assume that  $E_2^{k,l}\neq 0$. We apply Lemma 4.1 and Remark 4.4 to conclude that there exists a dominant weight $\gamma=p\delta + u\cdot 0$ with
$\Hom_{G}(V(\ga), \opH^{l}(G_1,H^0(\la))^{(-1)})\neq 0$ and $\Ext^{k}_G(V(\la)^{(r-1)}, H^0(\ga))\neq 0.$ 
Furthermore, by Propositions 4.2 and  4.3
there exists a sequence of non-zero weights
 $\la = \ga_0, \ga_1, \dots, \ga_{r-1} , \ga_r= \la \in X(T)_+$  with $\ga_1 = \ga$ such that 
 $\ga_j = p \delta_j + u_j\cdot 0$ for some $u_j \in W$ and  nonzero $\delta_j \in X(T)_+$.  Note that the linkage principle forces all $\delta_j$ to be in the root lattice. Since none of the fundamental weights are contained in the root lattice,  $\langle \delta_j ,\ta^{\vee}\rangle \geq 2$. From (\ref{rbound}), we get
\begin{equation}\label{rbound-2}
k+l \geq  \left(\sum_{j=1}^r (p-1) \langle \delta_j, 
        \ta^{\vee}\rangle\right) - r\geq r(2p-3).
\end{equation}
For the first inequality in (\ref{rbound-2}) to be an equality, from 
Proposition 4.3, we must have  
$p\delta_j -\delta_{j-1} + u_j\cdot 0= ((l_j-l(u_{j-1}))/2)\ta$ 
(where $l_j$ is as in Proposition 4.2) and
$\langle-u_j\cdot 0,\ta^{\vee}\rangle = \ell(u_j) + 1$ for all $1\leq j \leq r$.
Further, for the second inequality to be an equality, clearly we must have 
$\langle \delta_j ,\ta^{\vee}\rangle = 2$.
It follows immediately that 
$E_2^{k,l} =0$ whenever 
$k+l < r(2p-3)$.

Looking  at $\Hom_{G}(V(\ga), \opH^{l}(G_1,H^0(\la))^{(-1)})
\cong \Hom_B(V(\ga),S^{\frac{l-2n+1}{2}}(\ul^*)\otimes\ta))$ one concludes that 
for $l\leq 2p-3$  the only possible weights $\ga$ of the form $p\delta + u\cdot 0$ with $\delta$ in the root lattice that  make the above expression non-zero are $\la$ and zero. Clearly the trivial module does not appear as a section in a good filtration  of  $\opH^{l}(G_1,H^0(\la))^{(-1)}$ while $H^0(\la)$ appears only once. Namely, in the case $l=2p-3$. The latter follows from the discussion in Section 6.1.
 
 The linkage principle now forces  $l \geq 2p-3$. Moreover, if $l=2p-3$, the only possible choice  is  $\gamma = \la$, and hence $E_2^{k,2p-3}
 \cong \Ext^{k}_G(V(\la)^{(r-1)}, H^0(\la)).$

If $k+l= r(2p-3)$, then (\ref{rbound-2}) becomes an equality. This forces $\ga _1=\ga=p\delta+u\cdot 0=((l-l(u))/2)\ta+ \ta$ and $\delta_1=\delta = \ta$. 
The only elements $u$ of the Weyl group with $u \cdot 0$ being a multiple of $\ta$ are the identity and $s_{\ta}$. Now $
\langle-u\cdot 0,\ta^{\vee}\rangle = 
\ell(u) + 1$ forces 
$\ga = p\omega_1+ s_{\ta}\cdot 0=\la$, which forces $l=2p-3$ and $k=(r-1)(2p-3)$.

As in 5.4 it follows that
$$
E_2^{k,l} \cong
 \begin{cases}
 0 &\text{ if } k+l < r(2p-3)\\
 0 &\text{ if } l < 2p-3\\
 0 &\text{ if } k+l = r(2p-3) \text{ and } l\neq 2p-3\\
\Ext^{k}_G(V(\la)^{(r-1)}, H^0(\la)) &\text{ if } k+l = r(2p-3) \text{ and } l=2p-3.
\end{cases}
$$ 
 Therefore, the $((r-1)(p-2),p-2)$-term of the $E_2$-page transgresses to the $E_{\infty}$-page and 
produces an isomorphism $\Ext^{r(p-2)}_G(V(\la)^{(r)}, H^0(\la))\cong \Ext^{(r-1)(p-2)}_G(V(\la)^{(r-1)}, H^0(\la))$, and the claim follows by
induction.
\end{proof}

\begin{rem} We have actually shown a stronger statement. Namely, for any dominant weight $\la$ of the form $p \delta + u \cdot 0$ with $\delta$ in the root lattice, one has 
$$
\Ext^{i}_G(V(\la)^{(r)}, H^0(\la))  \cong
 \begin{cases}
  0 &\text{ if } i < 2p-3,\\
 0 &\text{ if }i =  r(2p-3) \text{ and } \la \neq (p-n)(\omega_1+\omega_n),\\
 k &\text{ if }i =  r(2p-3) \text{ and } \la = (p-n)(\omega_1+\omega_n).
\end{cases}
$$ 

\end{rem}
From Theorem 2.8(B) one concludes the following.
\begin{cor}
Suppose $\Phi$ is of type $A_n$ with $n\geq 2$ and $p \geq n+2$. Then $\opH^{r(2p-3)}(\gfpr,k) \neq 0.
$
\end{cor}
Corollary 6.2 and Theorem 4.4 imply that the least positive $i$ with 
$\opH^i(\gfpr,k) \neq 0$ satisfies $r(p - 2) \leq i \leq r(2p - 3)$.  
In the following sections, we identify precisely 
the value of $i$.  The answer will depend on the relationship between
$p$ and $n$.


\subsection{\bf Counting Simple Roots.} Let $\alpha_1, \alpha_2,\dots, \alpha_n$ denote the simple roots and $\omega_1, \omega_2, ... , \omega_n$ the corresponding fundamental weights. 
Any weight $\gamma$ can be written in the form $\gamma = \sum_{j=1}^n c_j \alpha_j$ with $c_j \in {\mathbb Q}$. We define $M_j(\ga) := c_j $,  $M(\ga) := \max \{ c_j \}$, and $m(\ga) := \max \{ j \;|\; c_j = M(\ga)\}.$ In addition we set $N_j:=j(n+1-j).$  We make the following

\begin{obs} For $1 \leq j \leq n$ and $w \in W$,
\begin{itemize}
\item[(a)] $\omega_j = \frac{1}{n+1}(\text{ sum of all positive roots which contain } \al_j)$,
\item[(b)] $N_j$ is the number of positive roots in $\Phi$ which contain $\al_j$,
\item[(c)] $m(\omega_j) = j$,
\item[(d)] $M_{m(\omega_j)} (\omega_j)=M_{j} (\omega_j)=M(\omega_j)= \frac{N_j}{n+1}$,
\item[(e)] $M_j(2\rho) = N_j$,
\item[(f)] $M_j(2\rho) = M_j(-w_0\cdot 0) \geq M_j(-w\cdot 0)$.
\end{itemize}
\end{obs}

Suppose we have 
dominant weights $\la, \mu$ with $\la = p\delta_2 + w_2\cdot 0$,  $\mu = p\delta_1 + w_1\cdot 0$ 
and $\opH^i(G,H^0(\mu)\otimes H^0(\la^*)^{(1)}) \neq 0$ for some $i>0$.  
As in the proof of Proposition 3.4, 
$\la - \delta_1$ must be a weight of $S^{\frac{i - \ell(w_1)}{2}}(\ul^*)$.
Hence $M_{m(\delta_2)}(\la -\delta_1)\leq \frac{i - \ell(w_1)}{2}.$ 
Using $\la = p\delta_2 + w_2\cdot 0$, it follows that 
\begin{equation}\label{mcond}
2pM_{m(\delta_2)}(\delta_2) + 2M_{m(\delta_2)}(w_2\cdot 0)- 2M_{m(\delta_2)}(\delta_1)+ l(w_1) \leq i.
\end{equation}

Note that $M_{m(\delta_2)}(\delta_1)\leq M(\delta_1)=M_{m(\delta_1)}(\delta_1)$ and that 
$-M_{m(\delta_1)}(w_1\cdot 0) \leq l(w_1)$ (from Observation 2.2(B)). One obtains
\begin{equation}\label{mcond2}
2pM_{m(\delta_2)}(\delta_2) + 2M_{m(\delta_2)}(w_2\cdot 0)- 2M_{m(\delta_1)}(\delta_1) 
-M_{m(\delta_1)}(w_1\cdot 0)\leq i.
\end{equation}
Next assume that $\mu = \la$. Then (\ref{mcond2}) results in
$
2(p-1)M_{m(\delta_1)}(\delta_1) + M_{m(\delta_1)}(w_1\cdot 0)\leq i.
$
Using Observation (f) gives
\begin{equation}\label{mcond3}
2(p-1)M_{m(\delta_1)}(\delta_1) + M_{m(\delta_1)}(-2\rho)\leq i.
\end{equation}

Suppose now that $\delta_1 = \omega_j$ for $1 \leq j \leq n$. From Observations (d) and (e), (\ref{mcond3}) becomes

\begin{equation}\label{nj}
 2\left(\frac{p-1}{n+1}\right)N_j - N_j \leq i.
\end{equation}

However, we can say more than this.  Suppose that in fact $M_j(-w_1\cdot 0) = N_j$.
Then, when expressed as a sum of distinct positive roots, $-w_1\cdot 0$
contains all $N_j$ roots containing $\al_j$ and possibly some other positive
roots. In other words, $-w_1\cdot 0 = (n + 1)\omega_j + \si$ where $\si$ is
a sum of distinct positive roots not containing $\al_j$.  Then
$$
\la = p\omega_j + w_1\cdot 0 = p\omega_j - (n + 1)\omega_j - \si
        = (p - n - 1)\omega_j - \si \leq (p - n - 1)\omega_j.
$$
Hence, the only way $\la$ can be dominant is if $\si = 0$.  In other words, 
$\la = (p - n - 1)\omega_j$ and we have shown the following.

\begin{prop} Suppose that $\Phi$ is of type $A_n$ with $n \geq 2$ and 
$\la = p\omega_j + w\cdot 0 \in X(T)_+$ for $w \in W$ is a 
weight of $S^{\frac{i - \ell(w)}{2}}(\ul^*)\otimes \omega_j$.
Then
$$
i \geq \Big{[}2\left(\frac{p-1}{n+1}\right) - 1\Big{]}j(n + 1 - j)
$$
with equality possible if and only if $\la = (p - n - 1)\omega_j$.
\end{prop}

\begin{rem} The assumption that $\la = p\omega_j + w\cdot 0$ is a weight of 
$S^{\frac{i - \ell(w)}{2}}(\ul^*)\otimes \omega_j$ places restrictions on $p$ (and $n$).
Indeed, $\la - \omega_j  = (p - 1)\omega_j + w\cdot 0$
must lie in the root lattice.  But $w \cdot 0$ lies in the root lattice. 
Therefore, $(p-1)\omega_j$ must also lie in the root lattice.
However, for $1 \leq j \leq (n+1)/2$, to have $(p-1)\omega_j$ or symmetrically
$(p-1)\omega_{n+1-j}$ in the root lattice, we must have that $(n+1)$ 
divides $(p-1)j$.  
\end{rem}

Given the constraints noted in the remark, it is useful to rewrite the 
right hand side of the proposition as

\begin{equation}\label{redo}
\Big{[}2\left(\frac{p-1}{n+1}\right) - 1\Big{]}j(n + 1 - j) 
        = \big{[}2(p-1) - (n + 1)\big{]}j\left(1 - \frac{j}{n+1}\right).
\end{equation}


\subsection{\bf Larger Weights.} In this section, we will see that the only non-zero 
non-fundamental dominant weight $\la$ which can have $\opH^i(G,H^0(\la)\otimes H^0(\la^*)^{(1)}) 
\neq 0$ for $0 \leq i \leq 2p - 3$ is the weight $\la = (p - n)\ta$ considered in Section 6.1.
Indeed, observe that when $p \geq n + 2$, 
$2(p-1) +2(n-1)  \left [\frac{p-1}{n + 1}-1\right]> 2p-3$.

\begin{prop} Suppose that $\Phi$ is of type $A_n$ with $n \geq 2$ and $p \geq n + 2$. 
Let $\la = p\mu + w\cdot 0 \in X(T)_{+}$ for $w \in W$ with
$\langle\mu,\ta\rangle \geq 2$.  If $\la \neq (p - n)\ta$, then 
$\opH^i(G,H^0(\la)\otimes H^0(\la^*)^{(1)}) = 0$ for 
$$0 < i \leq 2(p-1) +2(n-1)\left[\frac{p-1}{n+1} -1\right].$$
\end{prop}

\begin{proof} Let $\la$ and $\mu$ be as given. Assume that $i \neq 0$ and $\opH^i(G,H^0(\la)\otimes H^0(\la^*)^{(1)}) \neq 0$. 
Using  our definition of $m(\mu)$ from Section 6.3 we write $\mu = a\omega_{m(\mu)} + \sigma$, where $\sigma$ is a sum  of 
fundamental weights other than  $\omega_{m(\mu)}$ and $a$ is a positive integer. Note that $\sigma$ is nonzero if $a =1$. 
Set $j = \min \{ m(\mu), n+1- m(\mu)\}.$  Note that 
$M_{m(\mu)} (\omega_l)\geq \frac{j}{n+1}$  for any $l \neq j$. We obtain the following inequality
$$M_{m(\mu)}(\mu) \geq 
\begin{cases} \frac{2 N_j}{n+1} &\mbox{ if } a\geq 2,\\\\
\frac{N_j +j}{n+1} &\mbox{ else.}
\end{cases}
 $$
Substituting the above into equation (\ref{mcond3}) with $\delta_1 = \mu$ yields
$$
i \geq 
\begin{cases}
2\left(\frac{p-1}{n+1}\right)(2N_j) - N_j &\mbox{ if } a \geq 2,\\\\
2\left(\frac{p-1}{n+1}\right)(N_j + j) - N_j &\mbox{ else. }
\end{cases}
$$
As a function of $j$, $N_j$ is increasing on the interval $(0, (n+1)/2)$. Therefore both of the above expressions are minimal when $j$ is as small as possible. In the first case $j=1$ is possible. However, since we are assuming that $\mu \neq \omega_1 + \omega_n$, we may assume that $j\geq 2$ for the second case. One obtains 
\begin{equation}
i \geq 
\begin{cases}
2\left(\frac{p-1}{n+1}\right)2n - n &\mbox{ if } a \geq 2,\\\\
2\left(\frac{p-1}{n+1}\right)(2(n-1) +2 ) - 2(n-1)=2\left(\frac{p-1}{n+1}\right)2n - 2(n-1) &\mbox{ else. }
\end{cases}
\end{equation}

Hence,
\begin{eqnarray*}
i &\geq& 4n \left(\frac{p-1}{n + 1}\right)-2(n-1)\\ &=& 4(n+1) \left(\frac{p-1}{n + 1}\right)- 4 \left(\frac{p-1}{n + 1}\right)-2(n-1)\\
&=& 2(p-1) +\left[2(n+1)-4\right]  \left(\frac{p-1}{n + 1}\right)-2(n-1) \\
&=& 2(p-1) +2(n-1) \left(\frac{p-1}{n + 1}\right)-2(n-1)\\
& =&2(p-1) +2(n-1)  \left [\frac{p-1}{n + 1}-1\right]. 
\end{eqnarray*}
\end{proof}

To determine sharp vanishing bounds, we need to consider the relationship between 
$p$ and $n$.  This will be done in the succeeding sections.


\subsection{\bf The Case: $\gcd(p-1,n+1) = 1$.} It follows from Remark 6.3 that under 
this assumption, the weight  $(p - 1)\omega_j$ does not lie in the root lattice for any $j$.  
Therefore, $\mu$ must be the sum of at least two (not necessarily distinct) 
fundamental dominant weights, and $\langle\mu,\ta^{\vee}\rangle \geq 2$. From Corollary 3.5(b), 
we conclude that $i \geq  2p - 3$. Combining this with Proposition 6.1, Proposition 6.4, and Theorem 2.8(A),
we obtain these sharp bounds.

\begin{thm} Suppose $\Phi$ is of type $A_n$ with $n \geq 2$, $p \geq n + 2$ and 
$\gcd(p - 1,n+1) = 1$.  Then
\begin{itemize}
\item[(a)] $\opH^i(\gfp,k) = 0$ for $0 < i < 2p - 3$;
\item[(b)] $\opH^{2p - 3}(\gfp,k) \cong k$.
\end{itemize}
\end{thm}


\subsection{\bf The Case: $1 < \gcd(p-1,n+1) < n + 1$.}  For convenience, set $g := \gcd(p-1,n+1)$.  
We investigate here the inequality in Proposition 6.3.  Note that since
$n+1$ does not divide $p - 1$, neither $(p-1)\omega_1$ nor
$(p-1)\omega_n$ lie in the root lattice.  So we restrict attention 
to $\omega_j$ with $1 < j < n$. As such, there is nothing to consider unless
$n \geq 3$. Without a loss of generality assume that $j \leq (n+1)/2$.  Furthermore, 
from Remark 6.3, we may assume that $j(p-1)$ is divisible by $(n + 1)$.

Consider the function 
$f(j) = \Big{[}2\left(\frac{p-1}{n+1}\right) - 1\Big{]}j(n + 1 - j)$,
which is a quadratic polynomial in the variable $j$. For our purposes, we want
to minimize $f(j)$. This evidently occurs when $j$ is minimal (for $j \leq (n + 1)/2$).
So we consider the case that $j$ is minimal such that $j(p-1)$ is divisible by $(n + 1)$.  This implies that, $n + 1 = gj$, where $g = \gcd(n+1, p-1)$.  With this subsitution,
using (\ref{redo}), the inequality in Proposition 6.3 may be rewritten as
\begin{equation}\label{redoi}
i \geq \big{[}2(p-1) - (n + 1)\big{]}j\left(1 - \frac{1}{g}\right).
\end{equation}

Suppose first that $n + 1 = gj$ is odd (and $n \geq 4$). Then, both
$g$ and $j$ must be odd. Therefore, $g \geq 3$ and $3 \leq j \leq (n + 1)/2$.  
Hence, 
$$
j\left(1 - \frac{1}{g}\right) \geq 3\left(1 - \frac13\right) = 2.
$$
Equation (\ref{redoi}) allows us to conclude that
\begin{align*}
i &\geq 2\big{[}2(p-1) - (n + 1)\big{]}\\
        &= 4(p-1) - 2(n + 1)\\
        &= 2(p-1) + 2(p-1) - 2(n + 1)\\
        &= 2(p-1) + 2(p - n - 2)\\
        &\geq 2(p-1) = 2p - 2
\end{align*}
since $p \geq n + 2$. So we get a bound on $i$ which is strictly larger than
$2p - 3$.

Consider now the case that $n + 1 = gj$ is even (and $n \geq 3$).  Since $p - 1$ is
even, $g$ is necessarily even. In particular, $g \geq 2$.  Since $g \neq n + 1$, we
also have $j \geq 2$.  Suppose first that $j \geq 4$. Then
$$
j\left(1 - \frac{1}{g}\right) \geq 4\left(1 - \frac12\right) = 2.
$$
And so the same argument as in the $n+1$-odd case would show that $i \geq 2p - 2$.

Suppose next that $j = 3$ and $g > 2$. Then $g \geq 4$.  Then,
$$
j\left(1 - \frac{1}{g}\right) \geq 3\left(1 - \frac14\right) = \frac94 > 2.
$$
We conclude that $i > 2p - 2$.

Suppose next that $j = 3$ and $g = 2$.  Then $n + 1 = 6$ and $p - 1 > n + 1 = 6$.
So $p \geq 11$ (as $p$ is prime).  Here we get
\begin{align*}
i & \geq \big{[}2(p-1) - (n + 1)\big{]}j\left(1 - \frac{1}{g}\right)\\
        &= \big{[}2(p-1) - 6\big{]}3\left(1 - \frac12\right)\\
        &= 3\big{[}(p-1) - 3\big{]}\\
        &= 3p - 12\\
        &= 2p - 3 + p - 9\\
        &\geq 2p - 1
\end{align*}
since $p \geq 11$.

Suppose next that $j = 2$, i.e., $n + 1 = 2g$.  Since $p - 1 > n + 1$ 
and $g$ divides $p - 1$, we must have $p - 1 \geq 3g$.  Write $p - 1 = (3 + m)g$
for an integer $m \geq 0$.  Here we get
\begin{align*}
i & \geq \big{[}2(p-1) - (n + 1)\big{]}j\left(1 - \frac{1}{g}\right)\\
        &= \big{[}2(p-1) - 2g\big{]}2\left(1 - \frac{1}{g}\right)\\
        &= 4(p-1) - \frac{4(p-1)}{g} - 4g + 4\\
        &= 2(p-1) + 2(p-1) - \frac{4(p-1)}{g} - 4g + 4\\
        &= 2(p-1) + 2(3 + m)g - 4(3 + m) - 4g + 4\\
        &= 2(p-1) + 2g - 8 + m(2g - 4)\\
        &\geq 2(p-1) + 2g - 8
\end{align*}
since $m \geq 0$ and $g \geq 2$.  If $g \geq 4$, then we conclude that 
$i \geq 2p - 2$.

However, if $g = 2$, we can only conclude that $i \geq 2p - 6$.
This happens when $n + 1 = gj = 4$ or $n = 3$. Note that for $n = 3$,
we either have $g = 2$ or $g = 4$ with the latter case falling
into the $n + 1$ divides $p - 1$ category.  The case of $n = 3$ will
be dealt with specifically in Section 6.11. We summarize our findings in 
the following proposition.

\begin{prop} Suppose $\Phi$ is of type $A_n$ with $n \geq 4$.  Suppose further
that $p > n + 2$ and $1 < \gcd(p - 1,n+1) < n + 1$.   
Let $\la = p\omega_j + w\cdot 0 \in X(T)_+$ for $2 \leq j \leq n-1$ 
and $w \in W$.  Then $\opH^i(G,H^0(\la)\otimes H^0(\la^*)^{(1)}) = 0$ 
for $0 < i \leq 2p - 3$.
\end{prop}


\subsection{\bf The Case: $p - 1 = n + 1$.} Under this condition, we can explicitly 
construct a weight $\la$ with $\opH^{p-2}(G,H^0(\la)\otimes H^0(\la^*)^{(1)}) \neq 0$.
Let $\la = p\omega_1 + w\cdot 0$ where $w = s_1s_2s_3\dots s_n$ with 
$s_i$ denoting the simple reflection corresponding to the $i$th simple root.  
Then $-w\cdot 0$ is the sum of all $n$ positive roots
containing $\al_1$.  In other words, $-w\cdot 0 = (n + 1)\omega_1$.
So $\la = p\omega_1 + w\cdot 0 = p\omega_1 - (n + 1)\omega_1 = \omega_1$.
Since $\ell(w) = n = p - 2$, by Proposition 3.2,
$$
\dim\opH^{p-2}(G,H^0(\omega_1)\otimes H^0(\omega_1^*)^{(1)}) = 
\sum_{u \in W} (-1)^{\ell(u)} P_{0}( u\cdot\omega_1 - \omega_1) = 1.
$$
Note that $\omega_1^* = \omega_n$ and one can similarly argue
that $\opH^{p - 2}(G,H^0(\omega_n)\otimes H^0(\omega_n^*)^{(1)}) \cong k$.

\begin{thm}  Suppose $\Phi$ is of type $A_n$ with $n \geq 2$ and $p - 1 = n + 1$. 
Then
\begin{itemize}
\item[(a)] $\opH^i(\gfp,k) = 0$ for $0 < i < p - 2$;
\item[(b)] $\opH^{p-2}(\gfp,k) \cong k\oplus k$.
\end{itemize}
\end{thm}

\begin{proof} Part (a) follows from Theorem 4.4.  For part (b), from the discussion
above, we know that $\opH^{p - 2}(G,H^0(\la)\otimes H^0(\la^*)^{(1)}) \cong k$
for $\la = \omega_1$ or $\la = \omega_n$.  We 
claim that if $\opH^{p - 2}(G,H^0(\la)\otimes H^0(\la^*)^{(1)}) \neq 0$ for 
a dominant weight $\la$, then $\la = \omega_1$ or $\omega_n$.  To see this, 
write $\la = p\mu + w\cdot 0$ for some $\mu \in X(T)_{+}$ and $w \in W$.
By Corollary 3.5(b), $\mu$ must be a fundamental dominant weight.  
Apply Proposition 6.3 with $p - 1 = n + 1$. The proposition gives that, for $1 < j < n$,
$$
i \geq j(n + 1 - j) = j(p - 1 - j) \geq 2(p - 1 - 2) = 2p - 6 = (p - 2) + (p - 4).
$$
For $n \geq 3$, $p = n + 2 \geq 5$ and so this gives $i > p - 2$.
Hence, $\mu = \omega_1$ or $\omega_n$.  In other words $\la = p\omega_1 + w\cdot 0$
or $\la = p\omega_n + w\cdot 0$, respectively.  From the proof of Corollary 3.5, we
observe that in order to have 
$\opH^{p - 2}(G,H^0(\la)\otimes H^0(\la^*)^{(1)}) \neq 0$, 
when $-w\cdot 0$ is expressed as a sum of distinct positive roots, one
of those roots must be $\ta = \al_1 + \al_2 + \cdots + \al_n$. From 
Observation 2.2(B), it follows that $\ell(w) \geq n = p - 2$.  
Applying Proposition 3.2, we see that
$$
\dim\opH^{p-2}(G,H^0(\la)\otimes H^0(\la^*)^{(1)}) = 
\sum_{u \in W}(-1)^{\ell(u)} P_0(u\cdot \la - \mu).
$$
One has non-zero cohomology only if $u\cdot \la - \mu = 0$ which can
only happen if $\la = \mu$, which gives the claim.

Since the only dominant weight less than $\omega_1$ or $\omega_n$ 
is the zero weight, $\opH^i(G,k) = 0$ for $i > 0$, and $\omega_1$ and
$\omega_n$ lie in different linkage classes,  the 
discussion in Section 2.8 gives part (b).
\end{proof}


\subsection{\bf The Case: $\gcd(p-1,n+1) = n + 1 < p - 1$.} The case $p - 1 = n + 1$ is excluded since that was dealt 
with in Section 6.7. Since $n + 1$ divides $p - 1$, $(p-1)\omega_j$ lies in the root lattice for all 
$1 \leq j \leq n$, and we need to allow $\la = p\omega_j + w\cdot 0$ for 
all $j$ in our general argument.

Write $p - 1 = d(n+1)$ for an integer $d \geq 2$.  We can 
rewrite the inequality in Proposition 6.3 (see also (\ref{redo})) as
\begin{align*}
i &\geq 2j(p-1) - 2j^2\left(\frac{p-1}{n+1}\right) - j(n+1) + j^2\\
        &= 2(p-1) + (2j - 2)(p-1) - 2j^2\left(\frac{p-1}{n+1}\right) - j(n+1) + j^2\\
        &= 2(p-1) + \big{[}(2j -2)(n+1) - 2j^2\big{]}\left(\frac{p-1}{n+1}\right) - j(n+1) + j^2\\
        &= 2(p-1) + \big{[}(2j -2)(n+1) - 2j^2\big{]}d - j(n+1) + j^2.
\end{align*}
For $j = 1$ (or $j = n$), this inequality allows for a value of $i < 2p - 3$.  
This will be discussed more in the next section.  For this section, we 
focus on the case $2 \leq j \leq n - 1$.  By default, we need $n \geq 3$.
For such $j$, the least value of the right hand side
above occurs when $j = 2$ (or $j = n - 1$).  Substituting $j = 2$, the above
inequality becomes
\begin{equation}\label{dcase}
i \geq 2(p-1) + (2n - 6)d - 2n + 2.
\end{equation}
Since $d \geq 2$, (\ref{dcase}) becomes
$$
i \geq 2(p-1) + 2n - 10.
$$
If $n \geq 5$, then $i \geq 2p - 2$.  If $n = 4$ with $d \geq 3$, then (\ref{dcase}) becomes
$$
i \geq 2(p-1) + (2*4 - 6)3 - 2*4 + 2 = 2p - 2.
$$
If $n = 4$ and $d = 2$, then $p - 1 = 2(4 + 1) = 10$ or $p = 11$ and we can only
say that $i \geq 2p - 6$.  This case will be considered in Section 6.12. For $n = 3$, 
notice that the value of $d$ is irrelevant in (\ref{dcase}).  Irrespective of $d$,
we conclude that $i \geq 2p - 6$.  This case will be discussed in Section 6.11.
We summarize the conclusions of this section in the following.

\begin{prop} Suppose $\Phi$ is of type $A_n$ with $n \geq 4$.  Suppose further
that $p > n + 2$ and $\gcd(p - 1,n+1) = n + 1$.  If $n = 4$, assume further
that $p \neq 11$. Let $\la = p\omega_j + w\cdot 0 \in X(T)_+$ for $2 \leq j \leq n-1$ 
and $w \in W$.  Then $\opH^i(G,H^0(\la)\otimes H^0(\la^*)^{(1)}) = 0$ for 
$0 < i \leq 2p - 3$.
\end{prop}


\subsection{\bf The Case: $\gcd(p-1,n+1) = n + 1 < p - 1$, continued.}  In this section we 
investigate the case of $\la = p\omega_1 + w\cdot 0$ (or symmetrically, $\la = p\omega_n + w\cdot 0$).

\begin{lem} Suppose that $\Phi$ is of type $A_n$
Let $\la = p\omega_1 + w\cdot 0 \in X(T)_+$ or
$\la = p\omega_n + w\cdot 0 \in X(T)_+$ and $\mu = p\omega_j+ v\cdot 0 \in X(T)_+$
for $1 \leq j \leq n$ and $v, w \in W$. Then
$\operatorname{H}^i(G, H^0(\mu) \otimes H^0(\la^*)^{(1)})=0$ for 
$0<i< (p-1-n)n-(n+1-j)j+\ell(v)$. 
\end{lem}

\begin{proof} We give the argument for $\omega_1$. An analogous argument works for $\omega_n$.
If $\la =p\omega_1+w\cdot 0$ is dominant, then a direct computation shows that
$w = s_{k}s_{k-1} ... s_1$ and $\la = (p-k-1)\omega_1+\omega_{k+1}$, 
where $0 \leq k \leq n$. Here we are using the conventions $s_0=1$ and  
$\omega_{n+1}=0$. Moreover, we have the following equations for the formal characters
$$\ch (V((p-k)\omega_1) \otimes V(\omega_{k})) = \ch V((p-k)\omega_1+ \omega_k) + 
\ch V((p-k-1)\omega_1+ \omega_{k+1}).$$
As discussed earlier, the module $\ind_B^G(S^m(\ul^*)\otimes\omega_j)$ has a 
good filtration. We conclude that 
\begin{eqnarray*} 
&&\dim\Hom_G(V(p\omega_1+w\cdot 0), \ind_B^G(S^m(\ul^*)\otimes\omega_j))
\\
&&\leq \dim \Hom_G(V((p-k)\omega_1) \otimes V(\omega_{k}), \ind_B^G(S^m(\ul^*)\otimes\omega_j))\\
&&= \dim \Hom_G(V((p-k)\omega_1) , \ind_B^G(S^m(\ul^*)\otimes\omega_j)\otimes H^0(\omega_{n+1-k})),
\end{eqnarray*} 
where $k = \ell(w)$.

Note that $\langle \omega_j -u(\omega_{k}), \alpha_l^{\vee} \rangle \geq -1$ for all $u \in W$ and $1 \leq l \leq n$. 
It follows from \cite{KLT} that 
$R^j\ind_{B}^{G}( S^m(\ul^*)\otimes (\omega_j +u(\omega_{n+1-k}))) =0$ for $j > 0$. Using the long exact sequence 
that one obtains from a $B$-filtration of $H^0(\omega_{n+1-k})$ one concludes that 
\begin{eqnarray*} 
&&\dim\Hom_G(V(p\omega_1+w\cdot 0), \ind_B^G(S^m(\ul^*)\otimes\omega_j))
\\
&&\leq \dim \Hom_G(V((p-k)\omega_1) , \ind_B^G(S^m(\ul^*)\otimes\omega_j)\otimes H^0(\omega_{n+1-k}))\\
&&= \frac{1}{|\mbox{Stab}_W(\omega_k)|} \sum_{u\in W} \dim \Hom_G(V((p-k)\omega_1), \ind_B^G(S^m(\ul^*)
\otimes\omega_j-u(\omega_{k}))).
\end{eqnarray*} 
The Weyl module $V((p-k)\omega_1)$ has one-dimensional weight spaces 
\cite[II 2.16]{Jan}.  A theorem of Kostant \cite[24.2]{Hum1} implies that  
$$\sum_{m\geq 0} \sum_{x\in W} (-1)^{\ell(x)} P_m(x \cdot((p-k)\omega_1) - \omega_j+u(\omega_k)) = 1.$$
According to \cite[3.8]{AJ} and \cite{KLT} we know that 
\begin{eqnarray*} &&\sum_{x\in W} (-1)^{\ell(x)} P_m(x \cdot((p-k)\omega_1) - \omega_j+u(\omega_k))  \\
&&= \dim \Hom_G(V((p-k)\omega_1), \ind_B^G(S^m(\ul^*)\otimes\omega_j-u(\omega_{k}))\geq 0, 
\end{eqnarray*}
for all $m \geq 0$. Clearly, for $x \neq 1$,  
$$\mbox{height}((p-k)\omega_1-\omega_j + u(\omega_k)) > \mbox{height}(x \cdot((p-k)\omega_1) - \omega_j+u(\omega_k)).$$ 
We conclude that 
$$\sum_{x\in W} (-1)^{\ell(x)} P_m(x \cdot((p-k)\omega_1) - \omega_j+u(\omega_k)) = \begin{cases} 1 \mbox{  if  } m = \mbox{ height}((p-k)\omega_1 -\omega_j + u(\omega_k))\\
0 \mbox{ else}.
\end{cases} $$
We have 
\begin{align*}
\mbox{ height}((p-k)\omega_1-\omega_j + u(\omega_k)) &\geq 
        \mbox{ height}((p-k)\omega_1-\omega_j -\omega_{n+1-k})\\
        &= \frac{(p-k)n}{2}- \frac{(n+1-j)j}{2}- \frac{(n+1-k)k}{2}.
\end{align*}
Hence,
$$\dim \Hom_G(V(p\omega_1+ w\cdot 0), \ind_B^G(S^m\ul^*\otimes \omega_j))=0, $$ 
if 
$$m <\frac{(p-k)n- (n+1-j)j-(n+1-k)k}{2}$$  
where  $k= \ell(w)$.

Setting $m=\frac{i-\ell(v)}{2}$ and applying Lemma 3.1 and (\ref{ind}) yields
$$\operatorname{H}^i(G, H^0(p\omega_1+w\cdot 0) \otimes H^0((p\omega_j+v\cdot 0)^*)^{(1)})=0$$  
for 
$$\frac{i-\ell(v)}{2} < \frac{(p-k)n-(n+1-j)j-(n+1-k)k}{2}.$$
Hence, one obtains vanishing for  
\begin{equation}\label{incond}
i < (p-k)n-(n+1-k)k-(n+1-j)j+\ell(v). 
\end{equation} 
Since the global minimum of $(p-k)n-(n+1-k)k$ on the closed interval $[1,n]$ occurs at $k=n$,
we get the claimed vanishing for 
$$
i \leq (p-1-n)n-(n+1-j)j+\ell(v).
$$
\end{proof}

\begin{prop} Suppose $\Phi$ is of type $A_n$ with $n \geq 3$.  Suppose further
that $p > n + 2$ and $\gcd(p - 1,n+1) = n + 1$. 
Let $\la = p\omega_1 + w\cdot 0 \in X(T)_+$ or $\la = p\omega_n + w\cdot 0$ 
with $w \in W$.  Then $\opH^i(G,H^0(\la)\otimes H^0(\la^*)^{(1)}) = 0$ for 
$0 < i \leq 2p - 3$.
\end{prop}

\begin{proof}   Equation (\ref{incond}) implies, for the case  $\la = \mu$,
 vanishing for $0 < i < 
 (p-k)n-(n+1-k)k-n+\ell(v)=(p-k)n-(n-k)k-n.$
 Again the global minimum occurs at $k=n$ and one obtains vanishing for $0<i\leq(p - 1 - n)n$. Observe that $p \geq 2(n+1)+1= 2n +3.$ Now
\begin{align*}
(p-1-n)n &= 2p-3 + p(n-2) - n(n+1)+3 \\
        &\geq 2p-3 +(2n+3)(n-2)-n^2-n+3 \\
        &= 2p-3 + (n-1)^2 -4 \\
        &\geq 2p-3
\end{align*}
with equality if and only if $n=3$ and $p=2n+3$. This case  does not occur. 
\end{proof}


\subsection{\bf The Case: $n = 2$.} Assume for this subsection that $\Phi$ is of type $A_2$ 
with $p > 3$. From Proposition 6.1 we know that $\opH^{2p-3}(\gfp, k) \neq 0.$ Proposition 6.5 
implies that $2p-3$ is indeed the lowest bound unless $3$ divides $p-1$. Note that 
the case $p - 1 = 3$ is not possible for a prime $p$.  If 3 divides $p -1$,
then the only possible non-zero cohomology in lower degrees would come from 
weights of the form $\la =p \omega_1 + w\cdot 0$ or the dual case 
$\la =p \omega_2 + w\cdot 0$. It follows from
Lemma 6.9 (see also the proof of Proposition 6.9) 
that $\opH^i(G, H^0(\la) \otimes H^0(\la^*)^{(1)}) =0$ for $i < 2p-6$. 
Moreover, using the arguments of Lemma 6.9 one can show that $\opH^i(G, H^0(\la) \otimes H^0(\la^*)^{(1)}) =0$ 
for $i\leq 2p-6$ unless $\la = (p-3)\omega_j$. In that case
$\sum_{u\in W} (-1)^{\ell(u)} P_{p-4}(u \cdot((p-3)\omega_j) - \omega_j)=1$ because the 
height of $(p-4)\omega_j$ is exactly $p-4$.
Proposition 3.2 now
says that $\opH^{2p-6}(G, H^0((p-3)\omega_j) \otimes H^0((p-3)\omega_{(3-j)})^{(1)}) \cong k$, $j=1,2$. 
Note that $\omega_1$ and $\omega_2$ are in different linkage
classes.  We conclude the following from Theorem 2.8(A) and the 
linkage discussion in Section 2.8.

\begin{prop}
Suppose $\Phi$ is of type $A_2$ and $p > 3$.
\begin{itemize}
\item[(a)]  If  $3$ divides $p-1$, then
\begin{itemize}
\item[(i)] $\opH^i(\gfp, k) = 0$ for $0 < i < 2p-6$;
\item[(ii)] $\opH^{2p-6}(\gfp, k) \cong k \oplus k.$ 
\end{itemize}
\item[(b)]  If  $3$ does not divide $p-1$, then
 $\opH^i(\gfp, k) = 0$ for $0 < i < 2p-3$. 
\end{itemize}
\end{prop}


\subsection{\bf The Case: $n = 3$.}
Let $\Phi$ be of type $A_3$ with $p > 4$. The case $p=5$ is included in 
Proposition 6.7.  For the remainder of this section we assume that $p >5$.

\begin{lem} Suppose that $\Phi$ is of type $A_3$ with $p > 4$. Then 
$$\sum_{u\in W} (-1)^{\ell(u)} P_{p-5}(u \cdot((p-4)\omega_2) - \omega_2)=1.$$
\end{lem}

\begin{proof} Observe that $2 \omega_2 = \alpha_1 + 2 \alpha_2 + \alpha_3= \epsilon_1 + \epsilon_2 - \epsilon_3 -\epsilon_4.$ For $u \cdot((p-4)\omega_2) - \omega_2$ to be a sum of positive roots one needs $u(2\omega_2)$ to be a sum of positive roots. This is the case if and only if either $u(2\omega_2) = 2 \omega_2$ or $u(2\omega_2) = \alpha_1 + \alpha_3.$ But we can rule out the second case because here $u \cdot((p-4)\omega_2) - \omega_2= ((p-5)/2) (\alpha_1 +\alpha_3) - \alpha_2 +u \cdot 0$. This is clearly not the sum of positive roots. It follows that we only have to consider $u \in\mbox{Stab}_W(\omega_2)= \{1, s_1, s_3, s_1s_3\}$ and that $u \cdot((p-4)\omega_2) - \omega_2 = (p-5)\omega_2 + u\cdot 0$.
A direct computation now shows that $$P_{p-5}((p-5)\omega_2 + u \cdot 0)= \begin{cases}
(p-5)/2 & \mbox{ if } u = 1\\
(p-7)/2 & \mbox{ if } u = s_1, s_3, s_1s_3.
\end{cases}$$ 
Hence, 
$\sum_{u\in W} (-1)^{\ell(u)} P_{p-5}(u \cdot((p-4)\omega_2) - \omega_2)=\sum_{u\in\mbox{Stab}_W(\omega_2)} (-1)^{\ell(u)} P_{p-5}(u \cdot((p-4)\omega_2) - \omega_2)=1.$ 
\end{proof}

By Proposition 6.1 we know that $\opH^{2p-3}(\gfp, k) \neq 0.$ 
Note that $\gcd(p-1,n+1) = \gcd(p-1,4) = 2$ or $4$.  If $\gcd(p-1,4) = 2$,
then by Remark 6.3, the only possible non-zero cohomology in lower degrees
would come from weights of the form $\la = p \omega_2 + w\cdot 0$.  On 
the other hand, if $\gcd(p-1,4) = 4$, then Proposition 6.9 gives the 
same conclusion.

Suppose $\la = p \omega_2 + w\cdot 0$.  It 
follows from Proposition 6.3 that $\opH^i(G, H^0(\la) \otimes H^0(\la^*)^{(1)}) =0$ 
for $i < 2p-6$ and that $\opH^{2p-6}(G, H^0(\la) \otimes H^0(\la^*)^{(1)}) =0$  
unless $\la = (p-4)\omega_2$.  That occurs if $w\cdot 0 = -4\omega_2$.  By 
direct computation one finds that $\ell(w) = 4$.  By Proposition 3.2 and Lemma 6.11 above, 
we have
\begin{align*}
\dim\opH^{2p-6}(G, H^0((p-4)\omega_2) &\otimes H^0((p-4)\omega_2^*)^{(1)})\\
&=  \sum_{u\in W} (-1)^{\ell(u)} P_{\frac{(2p - 6) - 4}{2}}(u \cdot((p-4)\omega_2) - \omega_2)\\
&=  \sum_{u\in W} (-1)^{\ell(u)} P_{p-5}(u \cdot((p-4)\omega_2) - \omega_2) =1.
\end{align*}
Combining this with Theorem 2.8(A) one obtains that 
$\opH^{2p-6}(\gfp, k) \cong k$. We summarize our findings for $\Phi=A_{3}$ below.

\begin{prop}
Suppose $\Phi$ is of type $A_3$ and $p > 4$.
\begin{itemize}
\item[(a)]  If $p=5$, then
\begin{itemize}
\item[(i)] $\opH^i(\gfp, k) = 0$ for $0 < i < p-2$;
\item[(ii)] $\opH^{p-2}(\gfp, k) \cong k \oplus k.$ 
\end{itemize}
\item[(b)]  If  $p>5$, then
\begin{itemize}
\item[(i)] $\opH^i(\gfp, k) = 0$ for $0 < i < 2p-6$;
\item[(ii)] $\opH^{2p-6}(\gfp, k) \cong k.$ 
\end{itemize}
\end{itemize}
\end{prop}


\subsection{\bf The Case: $n = 4$ and $p = 11$.}
Assume for this section that $\Phi$ is of type $A_4$ with $p = 11$.
Then $\gcd(p-1,n+1) = \gcd(10,5) = 5$ and $2p-3=19$.  It follows from Proposition 6.9 and 
Proposition 6.4 that the only non-zero cohomology in degrees lower than $19$ has to come 
from weights of the form  $\la= 11 \omega_2 +w\cdot 0= 6\omega_2$ or 
$\la = 11\omega_3 + w\cdot 0 = 6\omega_3$. In 
these cases we could potentially achieve non-vanishing in degree $18$. 
The weights $\omega_2$ and $\omega_3$ are dual. We give the argument for
$\omega_2$.  An analogous argument works for $\omega_3$.  
Here $\ell(w)= 6$.  According to Proposition 3.2, to have
$\opH^{18}(G, H^0((p-5)\omega_2) \otimes H^0((p-5)\omega_3)^{(1)})\neq 0$ one needs 
$\sum_{u\in W} (-1)^{\ell(u)} P_{6}(u \cdot(6\omega_2) - \omega_2)\neq 0.$
The lemma below (Lemma 6.12) rules this case out.

According to  Proposition 6.1, $\opH^{19}(\gfp, k) \neq 0.$  Again it  follows 
from Proposition 6.9 and Proposition 6.4 that any  cohomology in degree $19$ other then 
the one coming from Proposition 6.1 has to come from weights of the form  
$\la= 11 \omega_2 +w\cdot 0$ or $\la = 11\omega_3 + w\cdot 0$. 
Suppose $\la = 11\omega_2 + w\cdot 0$. Again, the $\omega_3$ case is analogous.
Proposition 3.2 implies that  cohomology in odd 
degrees has to come from  weights with corresponding $w \in W$ of odd length. Now 
equation (\ref{mcond}) shows that this can only happen if $\ell(w) =5$. Note that here 
$M_2(-w\cdot 0) = \ell(w)=5$. The only such weight is  $\la = 6 \omega_2 + 
\ta$.  By Proposition 3.2 it suffices to show that 
$\sum_{u\in W} (-1)^{\ell(u)} P_{7}(u \cdot(6 \omega_2 + \ta) - \omega_2)= 0.$
We conclude from  the lemma below that $\opH^i(G,H^0(\la)\otimes H^0(\la^*)^{(1)}) = 0$ for 
$0< i< 2p-3$ and further that $\opH^{2p - 3}(G,H^0(\la)\otimes H^0(\la^*)^{(1)}) = 0$
unless $\la = (p - n)\ta = 7\ta$.

\begin{lem} Suppose that $\Phi$ is of type $A_4$ with $p =11$. Then 
\begin{itemize}
\item[(a)]
$\sum_{u\in W} (-1)^{\ell(u)} P_{6}(u \cdot(6\omega_2) - \omega_2)=0,$
\item[(b)] 
$\sum_{u\in W} (-1)^{\ell(u)} P_{7}(u \cdot(6 \omega_2 + \ta) - \omega_2)= 0.$
\end{itemize}
\end{lem}

\begin{proof} (a) Observe that $5 \omega_2 = 3\alpha_1 + 6 \alpha_2 + 4\alpha_3+ 2\alpha_4= 3\epsilon_1 + 
3\epsilon_2 - 2\epsilon_3 -2\epsilon_4-2\epsilon_5.$ For $u \cdot(6\omega_2) - \omega_2$ to be a sum of 
positive roots one needs $u(5\omega_2)$ to be a sum of positive roots. This is the case if and only 
if either $u(5\omega_2) = 5 \omega_2$ or $u(5\omega_2) = 3\alpha_1 +\alpha_2+ 4\alpha_3+2\alpha_4.$ 
Let us consider the second case. Here $u$ is an element of the co-set $s_2\cdot\mbox{Stab}_W(\omega_2).$ 
Note that for any $u \in s_2\cdot\mbox{Stab}_W(\omega_2)$ the expression $-u \cdot 0$ contains some positive 
multiple of the root $\alpha_2$. Therefore  $u \cdot(6\omega_2) - \omega_2= 3\alpha_1+ 4\alpha_3+2\alpha_4+u\cdot 0$ is 
not a sum of positive roots. It suffices therefore to look at 
$\sum_{u\in \mbox{Stab}_W(\omega_2)} (-1)^{\ell(u)} P_{6}(5\omega_2+u \cdot 0).$ A straightforward but tedious 
calculation now shows that $\sum_{u\in \mbox{Stab}_W(\omega_2)} (-1)^{\ell(u)} P_{6}(5\omega_2+u \cdot 0)=0.$

(b) As in part (a) for $u\cdot(6 \omega_2 + \ta) -\omega_2$ to be the sum of positive roots one needs either $u \in 
\mbox{Stab}_W(\omega_2)$ or $u \in s_2\cdot\mbox{Stab}_W(\omega_2)$. However, the second case results in weights 
that cannot be written as sums of $7$ positive roots. As in part (a) it suffices therefore to look at 
$\sum_{u\in \mbox{Stab}_W(\omega_2)} (-1)^{\ell(u)} P_{7}(6\omega_2+u \cdot \ta)$ which can be shown to be zero.

Both parts of the lemma can also be readily checked by using computer software such as MAGMA \cite{BC,BCP}.
\end{proof}


\subsection{\bf Summary for $r=1$.} The following theorem addresses (1.1.1) and (1.1.2) for type $A_n$ when $r = 1$.

\begin{thm}
Suppose $\Phi$ is of type $A_n$ with $n \geq 2$.  Suppose further
that $p > n + 1$.
\begin{itemize}
\item[(a)] (Generic case) If $p > n+2$ and $n > 3$, then
\begin{itemize}
\item[(i)] $\opH^i(\gfp, k) = 0$ for $0 < i < 2p-3$;
\item[(ii)] $\opH^{2p-3}(\gfp, k) \cong k$.

\end{itemize}
\item[(b)]  If $p = n+2$, then
\begin{itemize}
\item[(i)] $\opH^i(\gfp, k) = 0$ for $0 < i < p-2$;
\item[(ii)] $\opH^{p-2}(\gfp, k) \cong k\oplus k.$ 
\end{itemize}
\item[(c)]  If $n =2$ and $3$ divides $p-1$, then
\begin{itemize}
\item[(i)] $\opH^i(\gfp, k) = 0$ for $0 < i < 2p-6$;
\item[(ii)] $\opH^{2p-6}(\gfp, k) \cong k \oplus k.$ 
\end{itemize}
\item[(d)]  If $n =2$ and $3$ does not divide $p-1$, then
\begin{itemize}
\item[(i)] $\opH^i(\gfp, k) = 0$ for $0 < i < 2p-3$;
\item[(ii)] $\opH^{2p-3}(\gfp, k) \cong k.$ 
\end{itemize}
\item[(e)]  If $n =3$ and $p > 5$, then
\begin{itemize}
\item[(i)] $\opH^i(\gfp, k) = 0$ for $0 < i < 2p-6$;
\item[(ii)] $\opH^{2p-6}(\gfp, k) \cong k.$ 
\end{itemize}
\end{itemize}
\end{thm}

\begin{proof} Let $\la = p\mu + w\cdot 0 \in X(T)_{+}$ for
$\mu \in X(T)_{+}$ and $w \in W$.  Based on the discussion in 
Section 2, our goal has been to determine the least $i > 0$ 
such that 
$$\opH^i(G,H^0(\la)\otimes H^0(\la^*)^{(1)}) \neq 0.$$  
According to Proposition 6.1(a), we know that the weight 
$\la = p\ta - n\ta$ gives a non-zero cohomology class in degree $2p - 3$.

By Proposition 3.4(b), if $\langle\mu,\ta^{\vee}\rangle \geq 2$, then
$i \geq 2p - 3$.  Hence, the only way to obtain a smaller $i$ 
is for $\mu$ to be a fundamental weight.  That case has been dealt
with in previous sections, from which parts (a)(i), (b), (c), (d)(i),
and (e) follow.  It remains to show parts (a)(ii) and (d)(ii). From Proposition 6.4,
$\la = p\ta - n\ta = p(\omega_1 + \omega_n)$ is the only weight
with $\opH^{2p - 3}(G,H^0(\la)\otimes H^0(\la^*)^{(1)}) \neq 0$.
The result follows by Proposition 6.1(a) and Theorem 2.8(A).
\end{proof}



\subsection{\bf Results for $r>1$.}
The following theorem addresses (1.1.1) and (1.1.2) for type $A_n$ when $r > 1$ 
 and $p > 2n-2$. For $n >3$,  
a generic vanishing bound of degree $r(2p - 3)$ can be observed.

For $n + 2 \leq p \leq 2(n + 1)$, 
the methods employed in this paper should 
allow one to obtain precise vanishing bounds.  However,
given the number of special cases encountered in the $r = 1$ case, one would
expect even more non-generic behavior for $r > 1$. For example, it is easily  
seen that in the case $p=n+2$ non-vanishing already occurs in degree $r(p-2)$, 
i.e., $\opH^{r(p-2)}(\gfpr, k) \neq 0$. To give a complete answer many case-by-case
arguments will be necessary, most of them rather lengthy and intricate. 
For brevity we limit ourselves here to the  case where p is larger than twice the Coxeter number.

\begin{thm}
Suppose $\Phi$ is of type $A_n$ with $n \geq 2$ and  $p > 2(n + 1)$.   Then

\begin{itemize}
\item[(a)] (Generic case) If  $n > 3$, then
\begin{itemize}
\item[(i)] $\opH^i(\gfpr, k) = 0$ for $0 < i < r(2p-3)$;
\item[(ii)] $\opH^{r(2p-3)}(\gfpr, k) \cong k$.
\end{itemize}
\item[(b)]  If $n =2$ and $3$ divides $q-1$, then
\begin{itemize}
\item[(i)] $\opH^i(\gfpr, k) = 0$ for $0 < i < r(2p-6)$;
\item[(ii)] $\opH^{r(2p-6)}(\gfpr, k) \cong k \oplus k.$ 
\end{itemize}
\item[(c)]  If $n =2$ and $3$ does not divide $q-1$, then
\begin{itemize}
\item[(i)] $\opH^i(\gfpr, k) = 0$ for $0 < i < r(2p-3)$;
\item[(ii)] $\opH^{r(2p-3)}(\gfpr, k) \cong k.$ 
\end{itemize}
\item[(d)]  If $n =3$, then
\begin{itemize}
\item[(i)] $\opH^i(\gfpr, k) = 0$ for $0 < i < r(2p-6)$;
\item[(ii)] $\opH^{r(2p-6)}(\gfpr, k)  \cong k.$ 
\end{itemize}
\end{itemize}
\end{thm}

\begin{proof} From Remark 6.2, 
 for  $\la = p \delta + u \cdot 0$ with $\delta$ in the root lattice, the following holds
\begin{equation}
\Ext^{i}_G(V(\la)^{(r)}, H^0(\la))  \cong
 \begin{cases}
  0 &\text{ if } i < 2p-3,\\
 0 &\text{ if }i =  r(2p-3) \text{ and } \la \neq (p-n)(\omega_1+\omega_n),\\
 k &\text{ if }i =  r(2p-3) \text{ and } \la = (p-n)(\omega_1+\omega_n).
\end{cases}
\end{equation} 
From now on assume that $\la = p \delta + u \cdot 0$ and that $\delta$ not in the root lattice. Our goal is to obtain results like (6.14.1) for this situation.

  Assume further that 
$\Ext^i_G(V(\la)^{(r)}, H^0(\la)) \neq 0$. By Proposition 4.2, there exists a sequence of non-zero weights
 $\la = \ga_0, \ga_1, \dots, \ga_{r-1} , \ga_r= \la \in X(T)_+$ and nonnegative integers $l_1, l_2, \dots, l_r$ such that 
\begin{itemize}
\item[(i)] $i= \sum_{j=1}^rl_j$,
\item[(ii)] $\Ext^{l_j}_G(V(\ga_j)^{(1)}, H^0(\ga_{j-1})) \neq 0$, for $1\leq j \leq r$, and 
\item[(iii)] $\ga_j = p \delta_j + u_j\cdot 0$ for some $u_j \in W$ and  nonzero $\delta_j \in X(T)_+$.
\end{itemize}
Note that none of the $\delta_j$ are contained in the root lattice.

Next we apply our discussion in Section 6.3 to the pair of weights $\ga_j, \ga_{j-1}$ and obtain from (\ref{mcond2})
\begin{equation}
l_j\geq 2pM_{m(\delta_j)}(\delta_j) + 2M_{m(\delta_j)}( u_j\cdot 0)- 2M_{m(\delta_{j-1})}(\delta_{j-1}) 
-M_{m(\delta_{j-1})}( u_{j-1}\cdot 0).
\end{equation}

It follows from $\delta_r= \delta_0$ and $u_r = u_0$ that
\begin{eqnarray*}
i = \sum_{j=1}^r l_j &\geq& \sum_{j=1}^r 2pM_{m(\delta_j)}(\delta_j) + 2M_{m(\delta_j)}( u_j\cdot 0)- 2M_{m(\delta_{j-1})}(\delta_{j-1}) 
-M_{m(\delta_{j-1})}( u_{j-1}\cdot 0)\\
&=& \sum_{j=1}^r 2(p-1)M_{m(\delta_j)}(\delta_j) + M_{m(\delta_j)}( u_j\cdot 0)\\
&\geq&
\sum_{j=1}^r 2(p-1)M_{m(\delta_j)}(\delta_j) + M_{m(\delta_j)}(-2\rho),
\end{eqnarray*}
where the last inequality follows from Observation 6.3(f). We define 
\begin{equation}
d_j := l_j +2M_{m(\delta_{j-1})}(\delta_{j-1}) -
2M_{m(\delta_j)}(\delta_j) -M_{m(\delta_j)}( u_j\cdot 0) 
+M_{m(\delta_{j-1})}( u_{j-1}\cdot 0).
\end{equation} 
Then $i = \sum_{j=1}^r d_j$ and $d_j \geq   2(p-1)M_{m(\delta_j)}(\delta_j) + M_{m(\delta_j)}(-2\rho).$ 

In order to show vanishing up to the claimed degrees it is sufficient to show that $\sum_{j=1}^r d_j \geq r(2p-3)$, ($r(2p-6)$, respectively). We will actually show that strict inequalities hold in all but very few special cases. These special cases will yield statements (b)(ii) and (d)(ii) of the theorem.

According to Observation 6.3(e), we know that $M_{m(\delta_j)}(-2\rho) = -N_{m(\delta_j)}.$
Moreover, using the arguments in the proofs of Propositions 6.2, 6.3, and 6.4, one obtains the following bounds (recall that $\delta_j$ does not lie
in the root lattice):
\begin{equation}
d_j \geq \begin{cases}
\Big{[}2\left(\frac{p-1}{n+1}\right) - 1\Big{]}(n+1-m(\delta_j))m(\delta_j) &\mbox{ if }  \delta_j = \omega_{m(\delta_j)},\\
2(p-1) +2(n-1)  \left [\frac{p-1}{n + 1}-1\right] &\mbox{ else.}
\end{cases}
\end{equation}
Note that, as observed in Section 6.4, the second expression is strictly larger than $2p - 3$
under the assumptions on $p$ and $n$. Therefore, the only way to possibly obtain a $d_j$ with $d_j\leq 2p-3$ ($2p-6$, respectively) occurs when $\delta_j$ is a single fundamental weight. 
\\\\
{\bf Case 1: } $\delta_j \in\{\omega_2, ... ,\omega_{n-1}\}, n>2$.
\\\\
The expression
$\Big{[}2\left(\frac{p-1}{n+1}\right) - 1\Big{]}(n+1-m(\delta_j))m(\delta_j)$
attains a minimum when $m(\delta_j) = 2$ or $n - 1$ (see Section 6.8).  Hence,
\begin{eqnarray*}
d_j &\geq&  
2(p-1) + (2n-6)\left(\frac{p-1}{n+1}\right) - 2n+2\\
&\geq&2(p-1) + 2(2n-6) - 2n+2
\\
&=&2(p-1) + 2n-10.
\end{eqnarray*}
If $n \geq 5$ one obtains $d_j \geq 2p - 2 > 2p-3$. For $n=4$ one can show that  $d_j > 2p-3$ whenever $p > 13$.  We will discuss the case $n=4$ and $p \in \{11, 13\}$ separately later. For $n=3$ one obtains $d_j\geq 2p-6$. 
\\\\
{\bf Case 2: } $\delta_j = \omega_1$ or $\omega_n$. 
\\\\
Here  the above methods  produce the lower bound\begin{equation}
d_j \geq 2(p-1) - 2\left(\frac{p-1}{n+1}\right) - n,
\end{equation}
which is not sufficient. Other methods have to be applied. We will distinguish two cases.
\vskip .25cm 
\noindent
{\bf Case 2.1:} 
$\delta_{j-1} $ is a fundamental weight. 
\vskip .15cm 
We apply Lemma 6.9 to obtain from equation (\ref{incond}) 
$$l_j \geq (p-k)n-(n+1-k)k-(n+1-s)s+\ell( u_{j-1}),$$
where $k = \ell( u_j) =  -M_1( u_j\cdot 0)$.  By (6.14.3)
$$d_j \geq l_j + 2M_s(\omega_s) - 2M_1(\omega_1)  -M_1( u_j\cdot 0) 
+M_s( u_{j-1}\cdot 0).
$$ From Observation 2.2(B) (see Section 6.3), we know that $\ell( u_{j-1}) \geq -M_s( u_{j-1}\cdot 0)$ and clearly $M_s(\omega_s)\geq M_1(\omega_1)$. This yields 
$$d_j \geq (p-k)n-(n+1-k)k +k-(n+1-s)s=(p-k)n-(n-k)k -(n+1-s)s.
$$
As a function of $k$, the above attains its minimum at $k=n$. 
Hence 
$$d_j \geq (p-n)n-(n+1-s)s=2p-3 +(n-2)p-n^2-(n+1-s)s+3.$$
Using the assumption that $p \geq 2n+3$, one obtains
\begin{eqnarray*}d_j &\geq& 2p-3 +(n-2)(2n+3)-n^2-(n+1-s)s+3\\
&=&2p-3 +n^2-n-(n+1-s)s-3 .
\end{eqnarray*}
As an integer function of  $s$, the above attains its minimum at  
$$s= \begin{cases}
(n+1)/2 &\text{ for $n$ odd,}\\
n/2 &\text{ for $n$ even.}
\end{cases}$$ 
One concludes that 
$$\displaystyle{d_j \geq 
 \begin{cases}
 2p-3 +\frac{3(n-1)^2}{4}-4 &\text{ for $n$ odd,}\\
2p-3 +\frac{3n(n-2)}{4}-3.&\text{ for $n$ even.}
\end{cases}}$$
\\
\\
For $n > 3$ this yields $d_j > 2p-3$, for $n=3$ this yields $d_j \geq 2p-4$, and for $n=2$ this yields $d_j \geq 2p-6$.

\vskip .25cm 
\noindent
{\bf Case 2.2:}  $\delta_{j-1}$ is not a fundamental weight. 
\vskip .15cm 
Here we will show that $d_j+ d_{j-1} \geq 2(2p-3)$ (in the generic case). 
If  $j=1$ we set $d_0=d_r$.  Note that the following argument makes sense because $\gamma_0 = \gamma_r$.  Recall that none of the $\delta_j$ are contained in the 
root lattice.
Therefore, (6.14.4) yields 
\begin{equation}d_{j-1} \geq 2(p-1) +2(n-1)  \left [\frac{p-1}{n + 1}-1\right].
\end{equation}
 Adding (6.14.5) and (6.14.6) produces
\begin{eqnarray*}
d_j + d_{j-1} &\geq& 2(2p-3)+2 +2(n-2)  \left(\frac{p-1}{n + 1}\right)-3n+2\\
	&\geq&  2(2p-3) +4(n-2) -3n+4 \\
&=&  2(2p-3) + n-4.
\end{eqnarray*}
If $n \geq 5$ one obtains $d_j +d_{j-1}> 2(2p-3)$. The same holds for $n=4$ and $p > 11$ (since the second inequality above is in fact strict).  For $n=4$ and $p = 11$,
one has $d_j + d_{j-1} = 2(2p - 3)$  only if 
$\delta_{j-1} \in \{ \omega_1 + \omega_2,  \omega_1 + \omega_3,  \omega_4 + \omega_2,  \omega_4 + \omega_3\}$. 
 We will discuss this case later. For all others weights we get a strict inequality. For $n=2$ and $n=3$ it follows that  $d_j+d_{j-1}\geq 2(2p-4)$.

Assume now that $n >3$. If $n=4$ assume in addition that $p>13$. From above, we see that $d_j > 2p - 3$ unless $\delta_j  = \omega_1$
or $\omega_n$ and $\delta_{j-1}$ is not a fundamental weight.
If $d_j > 2p - 3$ for each $1 \leq j \leq r$, we have
$$
i \geq \sum_{j = 1}^{r}d_j > r(2p - 3)
$$
and vanishing for positive degrees up to $r(2p-3)$. 

Suppose now that $d_j \leq 2p - 3$ for some $j$. Recall that none of $\delta_j$ are assumed to be in the root lattice.
Let $t$ be the largest such $j$ and suppose that $t > 1$. 
Then 
$$
i \geq \sum_{j = 1}^r d_j = \sum_{t + 1}^r d_j + d_t + d_{t - 1} + \sum_{j = 1}^{t - 2} d_j
> (r - t)(2p - 3) + 2(2p - 3) + \sum_{j = 1}^{t - 2} d_j.
$$
Consider the remaining sum $\sum_{j = 1}^{t - 2} d_j$, again identify the 
largest $j$ with $d_j \leq 2p - 3$, and repeat this decomposition.  Continuing 
in this manner the claim follows except possibly if we reach a point when the 
largest $j$ with $d_j \leq 2p - 3$ is $j = 1$. So we are done if $d_1 > 2p - 3$.

Suppose now that $\delta_1 = \omega_1$ or $\omega_n$ and $\delta_0$ is not 
a fundamental weight so that we could have $d_1 \leq 2p - 3$.  Note that since
$\delta_1$ is a fundamental weight, $d_2 > 2p - 3$.  Recall that $\delta_0 = \delta_r$.  
Therefore, we can use the 
Case 2.2 argument to show that $d_1 + d_r > 2(2p - 3)$.  Further, since $d_2 > 2p - 3$,
the above argument can successfully be used to show that
$\sum_{j = 2}^{r - 1}d_j > (r - 2)(2p - 3)$. Hence, 
$$
i \geq d_1 + d_r + \sum_{j = 2}^{r - 1}d_j > 2(2p - 3) + (r - 2)(2p - 3) = r(2p - 3).
$$

One concludes that (6.14.1)  holds for any dominant weight $\la$ as long as $n\geq 5$ or $n=4$ and $p>13$.


For $n=3$ the above cases show that $i = \sum_{i=1}^r d_j \geq r(2p-6)$ with equality only in the case that all $\delta_j = \omega_2$. Using the arguments in Section 6.11  
one concludes that 
for weights $\mu_1 = p\omega_2+w_1\cdot0$ and $\mu_2 = p\omega_2+w_2\cdot0$
$$\Ext^i_G(V(\mu_2)^{(1)}, H^0(\mu_1)) = 
\begin{cases}
0 &\text{ if } 0< i < 2p-6,\\
0 &\text{ if } i = 2p-6 \text{ and  not both } \mu_1, \mu_2 \text{ are equal to  }  (p-4)\omega_2,\\
k &\text{ if } i = 2p-6 \text{ and } \mu_1=\mu_2 = (p-4)\omega_2.
\end{cases}
$$
An argument similar to those in Sections 5.4 and 6.2 now yields
\begin{equation}
\Ext^i_G(V(\la)^{(r)}, H^0(\la)) = 
\begin{cases}
0 &\text{ if } 0< i < r(2p-6),\\
0 &\text{ if } i = r(2p-6)\text{ and } \la \neq (p-4) \omega_2,\\
k &\text{ if } i = r(2p-6) \text{ and } \la = (p-4) \omega_2.
\end{cases}
\end{equation}

If $n=2$ and $3$ does not divide $p^r-1$ then $\Ext^i_G(V(\la)^{(r)}, H^0(\la)) \neq 0$ forces $\la$ and all the $\gamma_j$ to be in the root lattice. By Remark 6.2 one obtains (6.14.1).

If $n=2$ and $3$  divides $p^r-1$, one concludes from the above discussion that  $i = \sum_{i=1}^r d_j \geq r(2p-6)$ with equality only in the case that all $\delta_j \in \{\omega_1, \omega_2\}$. 
A direct computation using  Lemma 6.9 now shows that for weights  $\mu_1 = p\omega_r+w_1\cdot0$ and $\mu_2 = p\omega_s+w_2\cdot0$ with $r,s \in \{1,2\}$ 
$$\dim \Ext^i_G(V(\mu_2)^{(1)}, H^0(\mu_1)) = 
\begin{cases}
0 &\text{ if } 0< i < 2p-6,\\
0 &\text{ if } i = 2p-6\\
& \text{ and not both } \mu_1, \mu_2 \text{ are in } \{(p-3)\omega_i \;|\; i=1,2\}, \\
\delta_{rs} &\text{ if } i = 2p-6, \text{ both } \mu_1, \mu_2 \in \{(p-3)\omega_i \;|\; i=1,2\} \\
&\text{ and  }p \equiv 1 \mod 3, \\
1-\delta_{rs} &\text{ if } i = 2p-6, \text{ both } \mu_1, \mu_2 \in \{(p-3)\omega_i \;|\; i=1,2\}
\\
&\text{ and  }p \equiv -1 \mod 3.
\end{cases}
$$
Arguments like the ones in Sections 5.4 and 6.2 now yield
\begin{equation}
\Ext^i_G(V(\la)^{(r)}, H^0(\la)) = 
\begin{cases}
0 &\text{ if } 0< i < r(2p-6),\\
0 &\text{ if } i = r(2p-6)\text{ and } \la \neq (p-3) \omega_i, \; i \in \{1, 2\},\\
k &\text{ if } i = r(2p-6)\text{ and } \la = (p-3) \omega_i, \; i \in \{1, 2\}.
\end{cases}
\end{equation}
Note that in the case when $3$ divides $p-1$ all $\delta_j $ are the same, while in the case that $3$ divides $p+1$ the $\delta_j$ alternate between $\omega_1$ and $\omega_2$.

That leaves only the cases $n=4$ and $p=11$ or $13$. We will show that (6.14.1) holds in these cases. We discuss $p=13$ first. Looking at Cases 1 through 2.2 one can see that $d_j\leq 2p-3$ occurs when $\delta_j= \omega_2$ or its dual weight $\omega_3$. We will discuss the case $\delta_j=\omega_2$. We distinguish two cases, namely, $\delta_{j-1}$ being the sum of at least two fundamental weights and $\delta_{j-1}$ being a fundamental weight. In the first situation one can use an argument similar to Case 2.2 to show that $d_j + d_{j-1} > 2(2p-3)$. We leave the details to the interested reader. In the second situation the weight $\delta_{j-1}=\omega_1$, because $p\delta_j -\delta_{j-1}$ has to be an element of the root lattice.  In our series of estimates for $i$ and $l_j$ in  (6.3.1), (6.3.2), and (6.14.2), we replaced $-2M_{m(\delta_j)}(\delta_{j-1})$ by $-2M_{m(\delta_{j-1})}(\delta_{j-1})$. The former clearly being greater than or equal to the latter. Without this substitution one can obtain a better estimate for $d_j$, namely 
$  2(p-1)M_{m(\delta_j)}(\delta_j) + M_{m(\delta_j)}(-2\rho)+2M_{m(\delta_{j-1})}(\delta_{j-1})-2M_{m(\delta_j)}(\delta_{j-1}).$ In our special case this yields 
$d_j \geq 2\cdot12\cdot6/5-6+2\cdot4/5-2\cdot3/5>23,$ as required.

When $n=4$ and $p=11$ we need to revisit Case 1 with $\delta_j = \omega_2$ or $\omega_3$. We proceed as in the $p=13$ case. If $\delta_{j-1}$ is the sum of more than one fundamental weight an argument similar to Case 2.2 will show that $d_j + d_{j-1} > 2(2p-3)$. If $\delta_{j-1}$ is a single fundamental weight then the prime forces $\delta_{j-1}=\delta_{j}$. Here Lemma 6.12 implies $d_j > 2p-3$. Details are left to the reader. 
The last remaining situation is when $\delta_j= \omega_1$ or $\omega_4$. Note that equality holds in (6.14.6) if and only if $\delta_{j-1} \in \{ \omega_1 + \omega_2,  \omega_1 + \omega_3,  \omega_4 + \omega_2,  \omega_4 + \omega_3\}$.  In any of these four cases the height argument used in Case 2.1 yields $d_j= 2p-4$. From (6.14.6) one obtains $d_{j-1}= 2p + 4$. Hence $d_j + d_{j-1} > 2(2p-3)$ and (6.14.1) holds for $n\geq4$.

The weights $\la = (p-n) \widetilde{\alpha}$ as well as $(p-4)\omega_2$ for $n=3$ are the lowest non-zero dominant weights in their linkage classes. Theorem 2.8(A), (6.14.1), and (6.14.7) imply parts (a), (c) and (d). Similarly,  a slight variation of Theorem 2.8(A) and (6.14.8) yield part (b).

\end{proof}


\subsection{\bf The General Linear Group $GL_n({\mathbb F}_{q})$.}
For convenience we assumed throughout the paper that $G$ is a simple algebraic group. However, most results can be generalized to split reductive groups such as  $GL_n(k)$. In particular Sections 2.3 through 2.8 and  Formula (3.2.1) are valid for this larger set of groups.  One can therefore argue as in Sections 6.1 and 6.2 and obtain the following:

\begin{prop} Suppose  $n \geq 1$, $r \geq 1$  and $p \geq n + 2$.
Let $\la = (p - n)\ta$, $\ta$ being the maximal positive root.
Then 
\begin{itemize} 
\item[(a)]$\opH^{r(2p - 3)}(GL_n(k),H^0(\la)\otimes H^0(\la^*)^{(1)}) \cong k$; 
\item[(b)] $\opH^{r(2p - 3)}(GL_n({\mathbb F}_{q}),k) \neq 0$.
\end{itemize}
\end{prop} 
In \cite[Appendix]{FP} Friedlander and Parshall found the sharp bound $r(2p-3)$ for the 
Borel subgroup $B(\mathbb{F}_{q})$ of $GL_n (\mathbb{F}_{q})$ for all odd $q$. 
Using the fact that the restriction map $\opH^{i}(GL_n({\mathbb F}_{q}),k) \to 
\opH^{i}(B({\mathbb F}_{q}),k)$ is injective one obtains the vanishing range 
of $0<i< r(2p-3)$ for the cohomology of $GL_{n}({\mathbb F}_{q})$. One can now 
combine this with the aforementioned proposition to obtain the following theorem 
which verifies the conjecture in Barbu \cite[Conjecture 4.11]{B} for $p\geq n+2$. 

\begin{thm}
Suppose $n \geq 1$.  
\begin{itemize}
\item[(a)] If $q$ is odd, then $\opH^i(GL_n({\mathbb F}_{q}), k) = 0$ for $0 < i < r(2p-3)$;
\item[(b)] If $p \geq n+2$, then $\opH^{r(2p-3)}(GL_n({\mathbb F}_{q}), k) \neq 0$.
\end{itemize}
\end{thm}


\let\section=\oldsection

\end{document}